\documentclass[12pt]{amsart}
\usepackage{amssymb}

\begin{document}

\newtheorem{theorem}{Theorem}[section]
\newtheorem{definition}[theorem]{Definition}
\newtheorem{lemma}[theorem]{Lemma}
\newtheorem{proposition}[theorem]{Proposition}
\newtheorem{problem}[theorem]{Problem}
\newtheorem{conjecture}[theorem]{Conjecture}
\newtheorem{question}[theorem]{Question}

\title{Free Bessel laws}
\author{T. Banica}
\address{T.B.: Department of Mathematics, Toulouse 3 University, 118 route de Narbonne, 31062 Toulouse, France. {\tt banica@picard.ups-tlse.fr}}
\author{S.T. Belinschi}
\address{S.T.B.: Department of Mathematics, University of Saskatchewan, 106 Wiggins Road, Saskatoon, SK S7N 5E6, Canada. {\tt belinsch@math.usask.ca}}
\author{M. Capitaine}
\address{M.C.: Department of Mathematics, Toulouse 3 University, 118 route de Narbonne, 31062 Toulouse, France. {\tt capitain@cict.fr}}
\author{B. Collins}
\address{B.C.: Department of Mathematics, Lyon 1 University, 43 bd du 11 novembre 1918, 69622 Villeurbanne, France and University of Ottawa, 585 King Edward, Ottawa, ON K1N 6N5, Canada. {\tt collins@math.univ-lyon1.fr}}


\begin{abstract}
We introduce and study a remarkable family of real probability measures $\pi_{st}$, that we call free Bessel laws. These are related to the free Poisson law $\pi$ via the formulae $\pi_{s1}=\pi^{\boxtimes s}$ and $\pi_{1t}=\pi^{\boxplus t}$. Our study includes: definition and basic properties, analytic aspects (supports, atoms, densities), combinatorial aspects (functional transforms, moments, partitions), and a discussion of the relation with random matrices and quantum groups.
\end{abstract}

\maketitle

\section*{Introduction}

In this paper we introduce and study a remarkable two-parameter family of real probability measures, that we call free Bessel laws. These appear naturally in the context of Voiculescu's free probability theory \cite{vdn}.

In free probability, the central role is played by Wigner's semicircle law:
$$\gamma=\frac{1}{2\pi}\sqrt{4-x^2}\,dx$$

This measure appears in the free version of the central limit theorem, in the same way as the Gaussian law appears in the classical case. Moreover, Wigner's result can be understood in this way. See \cite{vdn}.

An alternative approach is based on the analogy between the Poisson law and the Marchenko-Pastur law, also called free Poisson law:
$$\pi=\frac{1}{2\pi}\sqrt{4x^{-1}-1}\,dx$$

The free Bessel laws $\pi_{st}$ are natural two-parameter generalizations of $\pi$. They can be introduced in several ways, depending on the values of the parameters. In the case $s\in(0,\infty)$ and $t\in (0,1]$, which is the most important, $\pi_{st}$ appears as free compression of $\pi^{\boxtimes s}$ by a projection of trace $t$:
$$\pi_{st}=\left(\pi^{\boxtimes s}\right)_t$$

An alternative formula, which actually works for a larger class of parameters, makes use of both Voiculescu's free convolution operations:
$$\pi_{st}=\pi^{\boxtimes s-1}\boxtimes \pi^{\boxplus t}$$

This latter formula, while a bit less transparent than the first one, makes clear the relation with $\pi$. Indeed, we have the following particular cases:
$$\begin{cases}
\pi_{s1}=\pi^{\boxtimes s}\\
\pi_{1t}=\pi^{\boxplus t}
\end{cases}$$

The measure $\pi_{st}$ with $s\in\mathbb N$ appears as free analogue of the following measure, having as density a kind of $s$-dimensional Bessel function:
$$p_{st}=e^{-t}\sum_{p_1=0}^\infty\ldots\sum_{p_s=0}^\infty\frac{1}{p_1!\ldots p_s!}\,\left(\frac{t}{s}\right)^{p_1+\ldots+p_s}\delta\left(\sum_{k=1}^se^{2\pi ik/s}p_k\right)^s$$

The analogy between Bessel laws and free Bessel laws can be understood in several ways. For instance if $a_1,\ldots,a_s$/$\alpha_1,\ldots,\alpha_s$ are independent/free variables, each of them following the Poisson/free Poisson law of parameter $s^{-1}t$, then:
\begin{eqnarray*}
p_{st}&=&{\rm law}\left(\sum_{k=1}^se^{2\pi ik/s}a_k\right)^s\\
\pi_{st}&=&{\rm law}\left(\sum_{k=1}^se^{2\pi ik/s}\alpha_k\right)^s
\end{eqnarray*}

Summarizing, the free Bessel laws are natural generalizations of the free Poisson law, in connection with Voiculescu's operations $\boxtimes$ and $\boxplus$, and with the Bessel functions.  In this paper we perform a systematic study of these laws. 

The point is that the free Bessel laws have a number of remarkable combinatorial properties, coming from a subtle relation with several key objects: 
\begin{enumerate}
\item Poisson laws. We prove that the supports, atoms, densities, as well as the various functional transforms of $\pi_{st}$ are given by formulae similar to those for the free Poisson laws.

\item Noncrossing partitions. We prove that the combinatorics of $\pi_{st}$ with $s\in\mathbb N$ is encoded by the set $NC_s$ of noncrossing partitions having blocks of size multiple of $s$, studied by Edelman \cite{ed}, Stanley \cite{st}, Armstrong \cite{ar}. 

\item Random matrices. We prove that $\pi_{s1}$ with $s\in\mathbb N$ is the asymptotic law of $(DW)^s$, where $W$ is a Wishart matrix, and $D$ is a diagonal matrix formed by uniformly distributed $s$-roots of unity.

\item Quantum groups. We prove that $p_{st}$ with $s\in\mathbb N$ is related to the finite group $H_n^s=\mathbb Z_s\wr S_n$, and that $\pi_{st}$ is related to the free version of $H_n^s$. The relation is via asymptotic laws of truncated characters.
\end{enumerate}

The quantum group results are part of a ``representation theory correspondence'', which is currently under construction. We have here the following table, collecting various results from \cite{bbc}, \cite{bc1}, \cite{bc2} and from the present paper:
\begin{center}
\begin{tabular}[t]{|l|l|l|l|l|}
\hline Lie group&Classical law&Quantum law\\
\hline\hline $O_n$&Gaussian&Semicircular\\
\hline $U_n$&Complex Gaussian&Circular\\
\hline $S_n$&Poisson&Free Poisson\\
\hline $H_n$&2-Bessel&Free 2-Bessel\\
\hline $H_n^s$&Bessel&Free Bessel\\
\hline
\end{tabular}
\end{center}
\medskip

The measures in this table are related by the general correspondence found by Bercovici and Pata in \cite{bp}. The table itself can be thought of as providing a concrete realization of the main particular cases of the correspondence.

The noncrossing partition and random matrix results seem to be as well part of some general correspondences, extending fundamental results about $\pi$. We hope to come back with more results in this sense, in some future work.

The paper is organized in four parts, as follows:
\begin{enumerate}
\item In 1-3 we discuss the construction of the free Bessel laws, and their basic analytic and combinatorial properties.

\item In 4-6 we discuss the relation with noncrossing partitions, the moment formula, and the random matrix models. 

\item In 7-9 we discuss some free additivity properties, the classical analogues, and the compound Poisson law interpretations. 

\item In 10-12 we discuss representation theory aspects, with the finite group model for $p_{st}$, and the free quantum group model for $\pi_{st}$.
\end{enumerate}

\subsection*{Acknowledgements.} This work was done with help from several institutions, and in particular we would like to thank the CNRS, NSERC, and the Fields Institute.

\section*{0. Notations}

Associated to a real probability measure having sequence of moments $m_1$, $m_2$, $m_3$, $\ldots$ are the following functional transforms:
\begin{enumerate}
\item Stieltjes transform: $f(z)=1+m_1z+m_2z^2+\ldots$
\item $\psi$ transform: $\psi(z)=f(z)-1$.
\item $\chi$ transform: $\psi(\chi(z))=z$.
\item $S$ transform: $S(z)=(1+z^{-1})\chi(z)$.
\item Cauchy transform: $G(\xi)=\xi^{-1}f(\xi^{-1})$.
\item $K$ transform: $G(K(z))=z$.
\item $R$ transform: $R(z)=K(z)-1/z$.
\item $\eta$ transform: $\eta(z)=1-1/f(z)$.
\item $\Sigma$ transform: $\Sigma(z)=S(z/(1-z))$.
\end{enumerate}

Here all the notations, except maybe for that of the Stieltjes transform, are the standard ones from the free probability literature \cite{vdn}, \cite{v3}, \cite{hp}, \cite{ns}.

\section{Definition, basic properties}

The origins of free probability theory go back to Voiculescu's noncommutative central limit theorem, where the Gaussian law is replaced by Wigner's semicircle law \cite{v1}. Since then, the analogy between the Gaussian law and the semicircle law has served as a guideline for the whole theory. See \cite{vdn}.

For the purposes of this paper, the guiding analogy will be that between the Poisson law and the Marchenko-Pastur law, also called free Poisson law.

We recall that the Poisson law is the following probability measure:
$$p=\frac{1}{e}\sum_{r=0}^\infty \frac{\delta_r}{r!}$$

According to general results in free probability, the free analogue of this measure can be introduced in the following way.

\begin{definition}
The free Poisson law is given by:
$$\pi=\frac{1}{2\pi}\sqrt{4x^{-1}-1}\,dx$$
\end{definition}

The support of this measure is the interval where the square root is real, namely $[0,4]$. This measure is also called Marchenko-Pastur law. See \cite{v3}.

We denote by $\boxplus$ and $\boxtimes$ the free additive and multiplicative convolutions, and we use Voiculescu's $R$ and $S$ transforms, which linearize them. See \cite{v1}, \cite{v2}. 

Given a real probability measure $\mu$, one can ask whether the convolution powers $\mu^{\boxtimes s}$ and $\mu^{\boxplus t}$ exist, for various values of $s,t>0$. The problem makes sense, because of the one-to-one correspondence between measures and their transforms. More precisely, the question is whether $S_s(z)=S(z)^s$ and $R_t(z)=tR(z)$ are the $S$ and $R$ transforms of some real probability measures.

For the free Poisson law, the answer to these questions is well-known. We include here the precise statement, along with a complete proof. This will serve as a model for some subsequent generalizations.

\begin{theorem}
The measures $\pi^{\boxtimes s}$, $\pi^{\boxplus t}$ exist for any $s,t>0$.
\end{theorem}

\begin{proof}
The free Poisson law $\pi$, as introduced in Definition 1.1, is the $t=1$ particular case of the free Poisson law of parameter $t$, given by:
$$\pi_t=\max (1-t,0)\delta_0+\frac{\sqrt{4t-(x-1-t)^2}}{2\pi x}\,dx$$

The Cauchy transform of this measure is given by:
$$G(\xi)=\frac{(\xi+1-t)+\sqrt{(\xi+1-t)^2-4\xi}}{2\xi}$$

We can compute now the $R$ transform, by proceeding as follows:
\begin{eqnarray*}
\xi G^2+1=(\xi+1-t)G
&\implies&Kz^2+1=(K+1-t)z\\
&\implies&Rz^2+z+1=(R+1-t)z+1\\
&\implies&Rz=R-t\\
&\implies&R=t/(1-z)
\end{eqnarray*}

This expression being linear in $t$, the measures $\pi_t$ form a semigroup with respect to free convolution. Thus we have $\pi_t=\pi^{\boxplus t}$, which proves the second assertion.

Regarding now the measure $\pi^{\boxtimes s}$, there is no explicit formula for its density. However, we can prove that this measure exists, by using some abstract results. 

We have the following computation for the $S$ transform of $\pi_t$:
\begin{eqnarray*}
\xi G^2+1=(\xi+1-t)G
&\implies&zf^2+1=(1+z-zt)f\\
&\implies&z(\psi+1)^2+1=(1+z-zt)(\psi+1)\\
&\implies&\chi(z+1)^2+1=(1+\chi-\chi t)(z+1)\\
&\implies&\chi(z+1)(t+z)=z\\
&\implies&S=1/(t+z)
\end{eqnarray*}

In particular at $t=1$ we have $S(z)=1/(1+z)$, so the $\Sigma$ transform of $\pi$, which is by definition $\Sigma(z)=S(z/(1-z))$, is given by:
$$\Sigma(z)=1-z$$

The existence of $\pi^{\boxtimes s}$ follows now from general results in \cite{bv}. Indeed, it is shown there that the $\Sigma$ transforms of the probability measures which are $\boxtimes$-infinitely divisible are the functions of the form $\Sigma(z)=e^{v(z)}$, where $v:\mathbb C-[0,\infty)\to\mathbb C$ is analytic, satisfying $v(\bar{z})=\bar{v}(z)$ and $v(\mathbb C^+)\subset\mathbb C^-$ (here, and in what follows, we denote by $\mathbb C^+$ and $\mathbb C^-$ the upper and lower half-plane). 

In the case of the free Poisson law, the function $v(z)=\log (1-z)$ satisfies all the above properties, and this gives the result.
\end{proof}

The starting point for the considerations in the present paper is the following remarkable identity.

\begin{theorem}
For $s\geq 1$ and $t\in (0,1]$ we have:
$$\pi^{\boxtimes s-1}\boxtimes\pi^{\boxplus t}
=((1-t)\delta_0+t\delta_1)\boxtimes\pi^{\boxtimes s}$$
\end{theorem}

\begin{proof}
We know from the previous proof that the $S$ transform of $\pi$ is given by $S(z)=1/(1+z)$, and that the $S$ transform of $\pi^{\boxplus t}$ is given by $S(z)=1/(t+z)$. Thus the measure on the left has the following $S$ transform:
$$S(z)=\frac{1}{(1+z)^{s-1}}\cdot\frac{1}{t+z}$$

The $S$ transform of $\alpha_t=(1-t)\delta_0+t\delta_1$ can be computed as follows:
\begin{eqnarray*}
f=1+tz/(1-z)
&\implies&\psi=tz/(1-z)\\
&\implies&z=t\chi/(1-\chi)\\
&\implies&\chi=z/(t+z)\\
&\implies& S=(1+z)/(t+z)
\end{eqnarray*}

This shows that the measure on the right has the following $S$ transform:
$$S(z)=\frac{1}{(1+z)^s}\cdot\frac{1+z}{t+z}$$

Thus the $S$ transforms of our two measures are the same, and we are done.
\end{proof}

We are now in position of introducing a remarkable two-parameter family of real probability measures. We call them free Bessel laws, because of a certain relationship with the Bessel functions, to be discussed later on.

\begin{definition}
The free Bessel law is the real probability measure $\pi_{st}$ with $(s,t)\in (0,\infty)\times(0,\infty)-(0,1)\times (1,\infty)$, defined as follows:
\begin{enumerate}
\item For $s\geq 1$ we set $\pi_{st}=\pi^{\boxtimes s-1}\boxtimes\pi^{\boxplus t}$.
\item For $t\leq 1$ we set $\pi_{st}=((1-t)\delta_0+t\delta_1)\boxtimes\pi^{\boxtimes s}$.
\end{enumerate}
\end{definition}

The compatibility between (1) and (2) comes from Theorem 1.3.

We regard the free Bessel law $\pi_{st}$ as being a natural two-parameter generalization of the free Poisson law $\pi$, in connection with Voiculescu's free convolution operations $\boxtimes$ and $\boxplus$. Observe that we have the following formulae:
$$\begin{cases}
\pi_{s1}=\pi^{\boxtimes s}\\
\pi_{1t}=\pi^{\boxplus t}
\end{cases}$$

Concerning the precise range of the parameters $(s,t)$, the above results can be probably improved. The point is that the measure $\pi_{st}$ still exists for certain points in the critical rectangle $(0,1)\times (1,\infty)$, but not for all of them. 

We did a number of abstract or numeric checks in this sense, and the critical values of $(s,t)$ seem to form an algebraic curve contained in $(0,1)\times (1,\infty)$, having $s=1$ as an asymptote. However, the case we are the most interested in is $t\in (0,1]$, and here there is no problem: $\pi_{st}$ exists for any $s>0$.

\section{The measures $\pi_{s1}$}

In this section and in the next one we study the support, atoms and density of $\pi_{st}$. As in the $s=1$ case, the formulae depend on whether $t$ is bigger, smaller or equal to $1$. We start with a complete study in the $t=1$ case.

We will use several times a well-known result of Lindel\"of, stating that an analytic function $g:\mathbb C^+\to\mathbb C$ has nontangential limit $a$ at a point $x$ provided that $g(\mathbb C^+)$ omits at least two points of $\mathbb C$, and that we have $\lim_{t\to 1}g(\gamma(t))=a$ for a certain path $\gamma\subset\mathbb C^+$, tending to $x$ in the frontier of $\mathbb C^+$ as $t\to 1$.

\begin{theorem}
The measure $\pi_{s1}$ has the following properties:
\begin{enumerate}
\item There are no atoms.
\item The support is $[0,K]$ where $K=(s+1)^{s+1}/s^s$.
\item The density is analytic on $(0,K)$.
\item The density is bounded at $x=K$, and is $\sim 1/(\pi x^{s/(s+1)})$ at $x=0$. 
\end{enumerate}
\end{theorem}

\begin{proof}
We denote by $G,\eta$ the Cauchy and $\eta$ transforms of $\pi_{s1}$. We have:
$$G(1/z)=\frac{z}{1-\eta(z)}$$

We recall that at $s=1$ the eta transform is: 
$$\eta_1(w)=\frac{1-\sqrt{1-4w}}{2}$$

In the case $s=1$ the measure $\pi_{11}$ is the free Poisson law, and all the assertions are clear from the explicit formula of the density. At $s\neq 1$ we have 2 cases.

{\bf Case 1:} $s<1$. We know that $\eta$ is the right
inverse of $\Phi(w)=w(1-w)^s$. The function $\Phi$ is analytic on $\mathbb C-[1,\infty),$ and the derivative is:
$$\Phi'(w)=(1-w)^{s-1}(1-(s+1)w)$$ 

This derivative vanishes at $w=1/(s+1)$, and we have $\Phi(1/(s+1))=1/K$. Also, it follows that the right inverse $\eta$ extends analytically from $-\infty$ up to $1/(s+1)$. By \cite{bhb}, the restriction of $\eta$ to $\mathbb C^+$ extends continuously and injectively to $\mathbb R$. Thus $\eta(1/K,\infty)$ is a simple analytic curve in $\mathbb C^+$, tending to $\infty$ as $x\to\infty$. Now from the formula of $G$ we get the assertions (1), (2), (3). See \cite{bhb}.

By applying Lindel\"of's theorem to the function $g(z)=\eta(z)/z^{1/(s+1)}$, we get:
$$\lim_{x\to\infty}g(x)=e^{i\pi s/(s+1)}$$

The formula of the Cauchy transform tells us that:
\begin{eqnarray*}
\lim_{x\downarrow 0} G(x)x^{s/(s+1)} 
&=&\lim_{x\to\infty}G(x^{-1})x^{-s/(s+1)}\\
& = & \lim_{x\to\infty}\frac{x^{1/(s+1)}}{1-\eta(x)}\\
&=&e^{-i\pi s/(s+1)}
\end{eqnarray*}

Thus the negative imaginary part tends to $\sin (\pi s/(s+1))$, and we are done.

{\bf Case 2:} $s>1$. The statement refering to atoms follows from Case 1 and \cite{bhb}. It remains to show the corresponding statements for the density.

Consider the following function: $$\omega_s(z)=\eta(z)\left(\frac{z}{\eta(z)}
\right)^{1/s}$$

By \cite{bhb} the restriction of $\omega_s$ to $\mathbb C^+$ extends continuously and injectively to $\mathbb R$, we have $\eta_1\circ\omega_s=\eta$, and left inverse of $\omega_s$ is: 
$$\Phi_s(z)=z\left(\frac{z}{\eta_1(z)}\right)^{s-1}$$

The derivative of $\Phi_s$ is given by:
$$\Phi_s'(w)=\left(\frac{w}{\eta_1(w)}
\right)^{s-1}\left(1+(s-1)\left(1-\frac{w}{\eta_1(w)}\eta'_1(w)\right)\right)$$

As before, by using the formula of $\eta_1$, we get that:
\begin{enumerate}
\item $\Phi_s$ is increasing on $(-\infty,s/(s+1)^2)$.
\item $\Phi_s$ has a maximum at $s/(s+1)^2$ of value $1/K$.
\item $\Phi_s$ is decreasing to $1/2^{s+1}$ on $(s/(s+1)^2,1/4]$.
\end{enumerate}

As in Case 1, one shows that $\omega_s(-\infty,1/K)=(-\infty,s/(s+1)^2)$ and that $\omega_s(1/K,\infty)$ is a simple analytic curve in $\mathbb C^+$ tending to infinity when $x\to\infty$. The relation
$\eta_1\circ\omega_s=\eta$ and the
formula of $\eta_1$ guarantees that $\eta(1/K,\infty)
\subset\mathbb C^+$ and $\eta(-\infty,1/K)\subset\mathbb R$. We conclude that ${\rm supp\ }\pi_{s1}=[0,K]$ and that the density is analytic on the interior of the support. 

It remains to prove (4). We use the same method as in the previous case, along with subordination functions as a tool. For $x\in\mathbb R-\{0\}$ we have:
\begin{eqnarray*}
\frac{\omega_s(x)}{x^{2/(s+1)}}
&=&x^{1/s}\cdot\frac{\eta_1(\omega_s(x))^{(s-1)/s}}{\omega_s(x)^{(s-1)/(2s)}}\cdot\frac{\omega_s(x)^{(s-1)/(2s)}}{x^{2/(s+1)}}\\
& = & x^{1/s}\cdot\frac{\omega_s(x)^{(s-1)/(2s)}}{x^{2/(s+1)}}\left(
\frac{\eta_1(\omega_s(x))}{\sqrt{\omega_s(x)}}\right)^{(s-1)/s}
\end{eqnarray*}

Recall that $\omega_s(x)\to\infty$ as $x\to\infty$, so that the third
factor above has a finite limit in the closure of $\mathbb C^+$ (we use here the branch of the square root which is defined on $\mathbb C-\mathbb R^+$). We get from Lindel\"{o}f's theorem:
\begin{eqnarray*}
\lim_{x\to\infty}\frac{\eta_1(\omega_s(x))}{\sqrt{\omega_s(x)}}
&=&\lim_{w\to-\infty}\frac{\eta_1(w)}{\sqrt{w}}\\
&=&\lim_{w\to-\infty}\frac{1-\sqrt{1-4w}}{2i\sqrt{|w|}}\\
&=&i
\end{eqnarray*}

This gives the following formula:
\begin{eqnarray*}
\lim_{x\to\infty}\frac{\sqrt{\omega_s(x)}}{x^{1/(s+1)}}
&=&\lim_{x\to\infty}\left(\frac{\eta_1(\omega_s(x))}{
\sqrt{\omega_s(x)}}\right)^{(s-1)/s}\left(\frac{\sqrt{\omega_s(x)}}{x^{1/(s+1)}}\right)^{(s-1)/(2s)}\\
&=&i^{(s-1)/(2s)}\lim_{x\to\infty}\left(\frac{\sqrt{\omega_s(x)}}{x^{1/(s+1)}}\right)^{(s-1)/(2s)}
\end{eqnarray*}

We get from this equation that the limit is $i^{(s-1)/(s+1)}$, so:
\begin{eqnarray*}
\lim_{x\to\infty}\frac{\eta(x)}{x^{1/(s+1)}}
&=&\lim_{x\to\infty}\frac{\eta_1(\omega_s(x))}{\sqrt{\omega_s(x)}}
\cdot\frac{\sqrt{\omega_s(x)}}{x^{1/(s+1)}}\\
&=&i\cdot i^{(s-1)/(s+1)}\\
&=&e^{i\pi s/(s+1)}
\end{eqnarray*}

Together with the formula of $G$, this concludes the proof.
\end{proof}

\section{The measures $\pi_{st}$}

We discuss now the support, atoms and density of $\pi_{st}$, for general values of $s,t>0$. Recall first that at $s=1$ we have the following formula:
$$\pi_{1t}=\max (1-t,0)\delta_0+\frac{\sqrt{4t-(x-1-t)^2}}{2\pi x}\,dx$$

For general $s$ the density is no longer explicit, and the formula of the support is quite complicated. We fix $s,t>0$, and we make the following notations:
\begin{eqnarray*}
\Phi(w)&=&tw\left(\frac{1-w}{1-(1-t)w}\right)^s\\
w_\pm&=&\frac{ts-t+2\pm\sqrt{t^2(s-1)^2+4st}}{2(1-t)}\\
K_\pm&=&\frac{t}{\Phi(w_\mp)}
\end{eqnarray*}

With these notations, we have the following result.

\begin{theorem}
The measure $\pi_{st}$ with $t<1$ has the following properties:
\begin{enumerate}
\item The atomic part is $(1-t)\delta_0$.
\item The rest of the support is $[K_-,K_+]$.
\item The density is analytic on $(K_-,K_+)$.
\item The density is $0$ at both ends of the support.
\end{enumerate}
\end{theorem}

\begin{proof}
The starting point is the formula (2) in Definition 1.4. By using general results in \cite{ns} about compressions with free projections, we get:
$$\pi_{st}=(1-t)\delta_0+t
\left(\pi^{\boxtimes s}\right)^{\boxplus 1/t}\boxtimes\delta_t$$

We will first analyze the measure appearing on the right, namely:
$$\mu=\left(\pi^{\boxtimes s}\right)^{\boxplus 1/t}$$

We denote by $R,S,\ldots$ the various transforms of $\mu$, and by $R_1,S_1,\ldots$ the same functions, in the particular case $t=1$. With these notations, we have:
\begin{eqnarray*}
S_1(z)=\frac{1}{(1+z)^s}
&\implies&S(z)=\frac{t}{(1+tz)^s}\\
&\implies&\Sigma(w)=t\left(\frac{1-w}{1-(1-t)w}\right)^s
\end{eqnarray*}

We recognize in the last formula the function $\Phi$. Summarizing, we have proved that the $\eta$ transform of $\mu$ is the right inverse of $\Phi$.

Note that the restriction of $\Phi$ to $\mathbb C^+$ and to $\mathbb C^-$ extends continuously to $\mathbb R$, but the two extensions do not agree on $(1,1/(1-t))$ when $s\not\in\mathbb N$.

Let us analyse the derivative of $\Phi$. This is given by:
$$\Phi'(w)=\frac{t((1-t)w^2+(t-st-2)w+1)}{(1-(1-t)w)^2}\left(\frac{1-w}{1-(1-t)w}\right)^{s-1}$$

Thus, in the general case $s\notin\{2,3,\ldots\}$, in the domain of analyticity $\Phi$ has exactly two singularities, namely at the points $w_+$ and $w_-$.

In the case $s\in\{2,3,\ldots\}$, the points 
$1$ and $1/(1-t)$ become as well singularities, because the function becomes rational.

We have $w_-<1<1/(1-t)<w_+$, and by direct computation we get $\Phi(w_-)<\Phi(w_+)$. Moreover, $\Phi$ is analytic around infinity, and is a local diffeomorphism. This 
follows indeed from:
$$\left(\frac{1}{\Phi(w^{-1})}\right)'
=\frac{w^2-(2+st-t)w+(1-t)}{t(w-1)^2}\left(\frac{w-(1-t)}{w-1}\right)^{s-1}$$

Let us draw our conclusions about $\Phi$ from the above facts:
\begin{enumerate}
\item The restriction of $\Phi$ to $\mathbb R$ increases from $-\infty$
to $\Phi(w_-)$ on the interval $(-\infty,w_-)$,  then decreases from $\Phi(w_-)$ to zero on $(w_-,1)$.
\item The restriction of $\Phi$ to $\mathbb R$ decreases from $\infty$
to $\Phi(w_+)$ on the interval $(1/(1-t),w_+)$, then increases back to $\infty$ on $(w_+,\infty)$.
\item The curve $\Phi (1,1/(1-t))$ escapes in $\mathbb C-\mathbb R$ if $s\not\in
\mathbb N$, is $(-\infty,0)$ if $s$ is odd, and is $(0,\infty)$ if $s$ is even.
\item $\Phi$ is invertible on a neighbourhood $N$ 
of $\{\infty\}\cup(-\infty,w_-)\cup(w_+,\infty)$. Moreover, $N$ can be chosen so that
 $\Phi(N\cap\mathbb C^{\pm})\subset\mathbb C^\pm.$
\end{enumerate}

From \cite{bhb} we have $\eta(-\infty,0)=(-\infty,0)$, so we conclude that $\eta$ is the
unique inverse of $\Phi$ on $N$ which carries $\infty$ into itself. By using the particular form of $\Phi$ and the equation
$\Phi(\eta(z))=z$ we get that the restriction to $\mathbb C^+$ of $\eta$ 
extends continuously and injectively to $\mathbb R$. It is obvious that $\eta$ extends analytically along the real line from infinity 
all the way to $\Phi(w_\pm)$, and not to those points. 

{\bf Claim:} We have $\eta(\Phi(w_-),\Phi(w_+))
\subset\mathbb C^+$. 

We know that $\eta(\Phi(w_-),\Phi(w_+))
\subset\mathbb C^+\cup\mathbb R$ (we know $\eta(\Phi(w_-),\Phi(w_+))$
is a bounded set because $\eta$ is injective on a neighbourhood - in $\mathbb C\cup\{\infty\}$ - of infinity). 
Thus, we only need to show that $\eta(x)\not\in\mathbb R$ for any $x\in
(\Phi(w_-),\Phi(w_+))$. 

First, observe that the injectivity of $\eta$ on the real line, together with (1) and (2), forbids
$\eta(x)$ to belong to $[w_-,1]\cup[1/(1-t),w_+]$. Thus we need to worry
only about $(1,1/(1-t))$.  So, assume that we have $\eta(x)\in(1,1/(1-t))$.

{\bf Case 1:} $s\in 2\mathbb N+1$. The same injectivity property together 
with (3) above provides a direct contradiction, since $\eta(-\infty,0)=
(-\infty,0)$.

{\bf Case 2:} $s\notin\mathbb N$. We denote by $\eta^-(x)$ the value of the extension of 
$\eta$ from the lower half-plane. From $\eta(\bar z)=\overline{\eta(z)}$ we get $\eta(x)=\eta^-(x)$, so the following can happen only if $s$ is even, and we are done:
$$\lim_{z\to x,z\in\mathbb C^+}\Phi(\eta(z))=x=\lim_{z\to x,z\in\mathbb C^-}
\Phi(\eta(z))$$

{\bf Case 3:} $s\in 2\mathbb N$. In this case, $\Phi$ has
multiplicity $s$ around $1$ and $1/(1-t),$ i.e. on some small neighbourhoods of $0$ and $\infty$  respectively, $\Phi$ covers each
point with $s$ points in a neighbourhood of $0$ and $1/(1-t)$, respectively.
In particular, we claim that $(-\infty,0)$ has in each half-plane
$s/2$ bounded preimages via $\Phi$, which are  simple curves uniting $1$ and $1/(1-t)$. 

Indeed, we know that there is  some $\varepsilon<0$ so that $(\varepsilon,0)$ has a preimage $\gamma_\varepsilon$ starting from one and climbing into the upper half-plane. We will show that $\gamma_\varepsilon$ extends to a bounded curve $\gamma$ uniting
$1$ and $1/(1-t)$.

Choose $v>\Phi(w_+)$ and consider the intersection with
$\mathbb C^+$ of the circle of diameter $(0,v)$. Denote it by $C$. Then $
\eta(C)$ is a simple curve in the upper half-plane starting from zero
and ending to a point $\eta(v)\in(w_+,\infty).$ Since $\mu$ has no atoms, we conclude that 
$\eta(C)$ stays bounded away from $1$. We claim that $\gamma$ lies between 
$\eta(C)$ and $(0,\eta(v)).$ Indeed, otherwise it would have to intersect 
one of these two curves. If it were to intersect $(0,\eta(v))$, then we would have
that $\Phi$ maps a point from $(0,\infty)$ into a negative number, a contradiction (recall $s$ is even). If it were to intersect $\eta(C)$, then,
since $\Phi(\eta(C))=C$, we would have that $C$ intersects the negative half-line,
another contradiction. 

Thus, $\gamma$ is bounded. Since its image via $\Phi$ is unbounded, $\gamma$ must end at $1/(1-t)$. It is easy to argue that $\gamma$ is simple:
it could fail to be so only by meeting a critical point of $\Phi$. But there are
only two critical points except $1$ and $1/(1-t)$, namely $w_-$ and $w_+$, and neither
is mapped by $\Phi$ in $(-\infty,0]$. 

Now we complete our proof: we have seen that $\gamma$ separates $w_-$ and $w_+$ from $(1,1/(1-t))$. Thus, if the simple curve 
$L=\eta(\Phi(w_-),\Phi(w_+))$ uniting $w_-$ and $w_+$ 
were to touch $(1,1/(1-t))$, 
then $L$ would have to 
intersect $\gamma$, so that there were $a\in(\Phi(w_-),\Phi(w_+))$
with the property $a=\Phi(\eta(a))\in\Phi(\gamma)=(-\infty,0]$,
an obvious contradiction, since we have seen that $\Phi(w_-),\Phi(w_+)>0$.

The equation of $G$ together with the formula of $\pi_{st}$ in the beginning of the proof gives the result. 
\end{proof}

\begin{theorem}
The measure $\pi_{st}$ with $t>1$ has the following properties:
\begin{enumerate}
\item There are no atoms.
\item The support is $[0,K_+]$.
\item The density is analytic on $(0,K_+)$.
\item The density is $\sim 1/((t-1)x^{1/s})$ at the left endpoint of the support, and $0$ at the right endpoint.
\end{enumerate}
\end{theorem}

\begin{proof}
The absence of atoms follows from Theorem 2.1 and \cite{be}.

By using the same method as in the previous proof, we get that the $\eta$ transform of $\pi_{st}$ is the right inverse of the following function:
$$\Phi_{st}(w)=\frac{w(1-w)^s}{t+(1-t)w}$$

By taking the derivative we get that there is only one zero of $\Phi_{st}'$ in the domain of analyticity of $\Phi_{st}$ in the general case $s\not\in\mathbb N$, namely:
$$w_1=\frac{t(s+1)-\sqrt{t^2(s-1)^2+4st}}{2s(t-1)}$$

The other root is in fact right of $t/(t-1)$.
Thus, let us only worry about the value of $\Phi_{st}(w_1)$. This provides us with one over the right endpoint of the support. The left endpoint 
must now be zero. Indeed, for any $x>\Phi_{st}(w_1)$, $\eta_{\pi_{st}}(x)\not\in
[w_1,1]$ (in fact $\not\in(-\infty,1]$) by injectivity of $\eta_{\pi_{st}}$. If $s\not\in
\mathbb N$, it follows immediately that $\eta_{\pi_{st}}(x)\not\in(1,+\infty)$.
If $s\in\{2,3,\ldots\}$, we observe that $\Phi_{st}$, while being a rational function, has a singularity at infinity, because:
$$\left(\frac{1}{\Phi_{st}(w^{-1})}\right)'
=\left(\frac{w}{w-1}\right)^{s-1}\frac{tw^2-(t+st)w+ts-t}{(w-1)^2}$$

Thus $\eta$ is not analytic around $\infty$, so the support of $\pi_{st}$ touches $0$.

An argument similar to the one in proof of Theorem 2.1 shows that the support of $\pi_{st}$ is connected. It remains to study the behaviour of the density near zero. Following the idea of the previous proof, we get:
\begin{eqnarray*}
\lim_{x\to-\infty}\frac{\eta(x)}{x^{1/s}} 
&=&\lim_{w\to-\infty}\frac{\eta(\Phi_{st}(w))}{\Phi_{st}(w)^{1/s}}\\
&=&\lim_{w\to-\infty}w\left(\frac{w(1-w)^s}{t+(1-t)w}\right)^{-1/s}\\
&=&-\frac{1}{e^{i\pi/s}}\lim_{w\to-\infty}|w|^{1-1/s}\frac{(t+(t-1)|w|)^{1/s}}{1+|w|}\\
&=&(t-1)e^{i\pi(s-1)/s}
\end{eqnarray*}

An application of Lindel\"{o}f's theorem concludes the proof.
\end{proof}

\section{Noncrossing partitions}

In this section we find a combinatorial model, in terms of noncrossing partitions, for the Stieltjes transform of $\pi_{st}$ with $s\in\mathbb N$. This is obtained by generalizing a well-known result regarding the free Poisson laws, for which we refer to \cite{ns}.  Let us also mention that the $s=2$ case was already worked out, in \cite{bbc}.

\begin{proposition}
The Stieltjes transform of $\pi_{st}$ satisfies $f=1+zf^s(f+t-1)$.
\end{proposition}

\begin{proof}
We use formula of the $S$ transform in the proof of Theorem 1.3:
\begin{eqnarray*}
S=\frac{1}{(1+z)^{s-1}}\cdot\frac{1}{t+z}
&\implies&\chi=\frac{z}{(1+z)^s}\cdot\frac{1}{t+z}\\
&\implies&z=\frac{\psi}{(1+\psi)^s}\cdot\frac{1}{t+\psi}\\
&\implies&z=\frac{f-1}{f^s}\cdot\frac{1}{t+f-1}
\end{eqnarray*}

This gives the equation in the statement.
\end{proof}

In order to find a combinatorial interpretation of $f$ we use the sets $NC_s$, studied by Edelman \cite{ed}, Stanley \cite{st}, Armstrong \cite{ar}.

\begin{definition}
We use the following notations.
\begin{enumerate}
\item A partition of $\{1,\ldots,k\}$ is called noncrossing if the following happens: if $a\sim b$ and $x\sim y$ with $a<x<b<y$, then $a\sim x\sim b\sim y$. 
\item $NC_s(k)$ is the set of noncrossing partitions of $\{1,\ldots,sk\}$ into blocks of size multiple of $s$, and $NC_s$ is the disjoint union of the sets $NC_s(k)$.
\item The normalized length of a partition $p\in NC_s(k)$ is given by $k(p)=k$. Also, we denote by $b(p)$ the number of blocks of $p$.
\item In the above notations, we make the following convention: the value $k=0$ is allowed, with $NC_s(0)$ consisting of one element $\emptyset$, having $0$ blocks.
\end{enumerate}
\end{definition}

Probably most illustrating here is the following table, containing the diagrammatic description of the elements in $NC_s(k)$, for small
values of $s,k$.
\begin{center}
\begin{tabular}[t]{|l|l|l|l|l|}
\hline &$NC_1$&$NC_2$&$NC_3$\\
\hline $0$&$\ast$&$\ast$&$\ast$\\
\hline $1$&$|$&$\sqcap$&$\sqcap\hskip-.7mm\sqcap$\\
\hline $2$&$||$, $\sqcap$&$\sqcap\sqcap$, $\Cap$, $\sqcap\hskip-1.6mm\sqcap\hskip-1.6mm\sqcap$&$\sqcap\hskip-1.6mm\sqcap\sqcap\hskip-.7mm\sqcap$, $\dots$, $\sqcap\hskip-1.65mm\sqcap\hskip-1.6mm\sqcap
\hskip-1.6mm\sqcap\hskip-1.6mm\sqcap$ (4)\\
\hline $3$&$|||$, $|\sqcap$, $\sqcap |$, 
$\bigcap\hskip-5.6mm{\ }_{{\ }^{|}}$ \ , $\sqcap\hskip-.7mm\sqcap$&(12)&(22)\\
\hline $4$&$||||$, $||\sqcap$, $|\sqcap|$, $\dots$, $\sqcap\sqcap$, $\Cap$, 
$\sqcap\hskip-1.6mm\sqcap\hskip-1.6mm\sqcap$ (14)&(55)&(140)\\
\hline
\end{tabular}
\end{center}
\bigskip

In this table the asterisks represent empty partitions, the dots represent partitions which are not shown, and the numbers count the partitions.

With these notations, we have the following result.

\begin{theorem}
The Stieltjes transform of $\pi_{s1}$ with $s\in\mathbb N$ is given by
$$f(z)=\sum_{p\in NC_s}z^{k(p)}$$
where $k:NC_s\to\mathbb N$ is the normalized length.
\end{theorem}

\begin{proof}
With the notation $C_k=\# NC_s(k)$, the sum on the right is:
$$f(z)=\sum_kC_{k}z^k$$

For a given partition $p\in NC_s(k+1)$ we can consider the last $s$ legs of the first block, and make cuts at right of them (see \cite{bbc} for $s=2$). This gives a decomposition of $p$ into $s+1$ partitions in $NC_s$, and we get:
$$C_{k+1}=\sum_{\Sigma k_i=k}C_{k_0}\ldots C_{k_s}$$

By multiplying by $z^{k+1}$ then summing over $k$ we get that the generating series of these numbers satisfies $f-1=zf^{s+1}$. But this is the same as the equation $f=1+zf^{s+1}$ of the Stieltjes transform of $\pi_{s1}$, and we are done.
\end{proof}

\begin{theorem}
The Stieltjes transform of $\pi_{st}$ with $s\in\mathbb N$ is given by:
$$f(z)=\sum_{p\in NC_s}z^{k(p)}t^{b(p)}$$
where $k,b:NC_s\to\mathbb N$ are the normalized length, and the number of blocks.
\end{theorem}

\begin{proof}
We denote by $F_{kb}$ the number of partitions in $NC_s(k)$ having $b$ blocks, and we set $F_{kb}=0$ for other integer values of $k,b$. All sums will be over integer indices $\geq 0$. With these notations, the sum on the right in the statement is:
$$f(z)=\sum_{kb}F_{kb}z^kt^b$$

The recurrence formula for the numbers $C_k$ in the previous proof becomes:
$$\sum_bF_{k+1,b}=\sum_{\Sigma k_i=k}\sum_{b_i}F_{k_0b_0}\ldots F_{k_sb_s}$$

In this formula, each term contributes to $F_{k+1,b}$ with $b=\Sigma b_i$, except for those of the form $F_{00}F_{k_1b_1}\ldots F_{k_sb_s}$, which contribute to $F_{k+1,b+1}$. We get:
\begin{eqnarray*}
F_{k+1,b}&=&\sum_{\Sigma k_i=k}\sum_{\Sigma b_i=b}F_{k_0b_0}\ldots F_{k_sb_s}\cr
&+&\sum_{\Sigma k_i=k}\sum_{\Sigma b_i=b-1}F_{k_1b_1}\ldots F_{k_sb_s}\cr
&-&\sum_{\Sigma k_i=k}\sum_{\Sigma b_i=b}F_{k_1b_1}\ldots F_{k_sb_s}
\end{eqnarray*}

This gives the following formula for the polynomials $P_k=\sum_bF_{kb}t^b$:
$$P_{k+1}=\sum_{\Sigma k_i=k}P_{k_0}\ldots  P_{k_s}+(t-1)\sum_{\Sigma k_i=k}P_{k_1}\ldots P_{k_s}$$

In terms of $f=\sum_kP_kz^k$, we get the following equation:
$$f-1=zf^{s+1}+(t-1)zf^s$$

But this is the same as the equation $f=1+zf^s(f+t-1)$ of the Stieltjes transform of $\pi_{st}$, and we are done.
\end{proof}

\section{Moments}

We compute now the moments of $\pi_{st}$, for arbitrary values of the parameters $s,t$. We use the following method:
\begin{enumerate}
\item For $s\in\mathbb N$ the moments can be found by counting partitions.
\item For $s>0$ we have the same formula, by a complex variable argument.
\end{enumerate}

The moments can be expressed in terms of generalized binomial coefficients. We recall that the coefficient corresponding to $\alpha\in\mathbb R$, $k\in\mathbb N$ is:
$$\begin{pmatrix}\alpha\cr k\end{pmatrix}=\frac{\alpha(\alpha-1)\ldots(\alpha-k+1)}{k!}$$

We denote by $m_1,m_2,m_3,\ldots$ the moments of a given probability measure.

\begin{theorem}
The moments of $\pi_{s1}$ with $s>0$ are the Fuss-Catalan numbers:
$$m_k=\frac{1}{sk+1}\begin{pmatrix}sk+k\cr k\end{pmatrix}$$
\end{theorem}

\begin{proof}
In the case $s\in\mathbb N$, we know from Theorem 4.3 that $m_k=\# NC_s(k)$. The formula in the statement follows by counting partitions, see \cite{ar}.

In the general case $s>0$, observe first that the Fuss-Catalan number in the statement is a polynomial in $s$:
$$\frac{1}{sk+1}\begin{pmatrix}sk+k\cr k\end{pmatrix}=\frac{(sk+2)(sk+3)\ldots(sk+k)}{k!}$$

Thus, in order to pass from the case $s\in\mathbb N$ to the case $s>0$, it is enough to check that the $k$-th moment of $\pi_{s1}$ is analytic in $s$. But this is clear from the equation  $f=1+zf^{s+1}$ of the Stieltjes transform of $\pi_{s1}$.
\end{proof}

\begin{theorem}
The moments of $\pi_{st}$, $s>0$ are the Fuss-Narayana polynomials:
$$m_k=\sum_{b=1}^k\frac{1}{b}\begin{pmatrix}k-1\cr b-1\end{pmatrix}\begin{pmatrix}sk\cr b-1\end{pmatrix}t^b$$
\end{theorem}

\begin{proof}
In the case $s\in\mathbb N$, we know from Theorem 4.4 that $m_k=\sum_bF_{kb}t^b$, where $F_{kb}$ is the number of partitions in $NC_s(k)$ having $b$ blocks. The formula in the statement follows by counting such partitions, see \cite{ed}, \cite{st}.

This result can be extended to any $s>0$, by using a complex variable argument, as in the proof of Theorem 5.1.
\end{proof}

The Fuss-Catalan numbers are known to appear in several contexts, for instance as dimensions of the algebras introduced by Bisch and Jones in \cite{bj}. An important question here is to understand the meaning of the Fuss-Narayana numbers in this context. As explained in \cite{bbc}, this can be done at $s=1,2$, due to a natural correspondence between partitions in $NC_s$ and diagrams in $FC_s$. In the case $n=3,4,\ldots$ the situation is quite unclear, and we don't have an answer.

In the case $s\notin\mathbb N$, the moments of $\pi_{st}$ can be further expressed in terms of Gamma functions. We would like to work out here the case $s=1/2$.

\begin{proposition}
The moments  of $\pi_{1/2,1}$ are given by:
\begin{eqnarray*}
m_{2p}&=&\frac{1}{p+1}\begin{pmatrix}3p\cr p\end{pmatrix}\\
m_{2p-1}&=&\frac{2^{-4p+3}p}{(6p-1)(2p+1)}\cdot\frac{p!(6p)!}{(2p)!(2p)!(3p)!}
\end{eqnarray*}
\end{proposition}

\begin{proof}
The even moments of $\pi_{st}$ with $s=n-1/2$, $n\in\mathbb N$, are given by:
\begin{eqnarray*}
m_{2p}
&=&\frac{1}{(n-1/2)(2p)+1}\begin{pmatrix}(n+1/2)(2p)\cr 2p\end{pmatrix}\\
&=&\frac{1}{(2n-1)p+1}\begin{pmatrix}(2n+1)p\cr 2p\end{pmatrix}
\end{eqnarray*}

With $n=1$ we get the formula in the statement. Now for the odd moments, we can use here the following well-known identity:
$$\begin{pmatrix}m-1/2\cr k\end{pmatrix}=\frac{4^{-k}}{k!}\cdot\frac{(2m)!}{m!}\cdot\frac{(m-k)!}{(2m-2k)!}$$

With $m=2np+p-n$ and $k=2p-1$ we get:
\begin{eqnarray*}
m_{2p-1}
&=&\frac{1}{(n-1/2)(2p-1)+1}\begin{pmatrix}(n+1/2)(2p-1)\cr 2p-1\end{pmatrix}\\
&=&\frac{2}{(2n-1)(2p-1)+2}\begin{pmatrix}(2np+p-n)-1/2\cr 2p-1\end{pmatrix}\\
&=&\frac{2^{-4p+3}}{(2p-1)!}\cdot\frac{(4np+2p-2n)!}{(2np+p-n)!}\cdot\frac{(2np-p-n+1)!}{(4np-2p-2n+3)!}
\end{eqnarray*}

In particular with $n=1$ we get:
\begin{eqnarray*}
m_{2p-1}
&=&\frac{2^{-4p+3}}{(2p-1)!}\cdot\frac{(6p-2)!}{(3p-1)!}\cdot\frac{p!}{(2p+1)!}\\
&=&\frac{2^{-4p+3}(2p)}{(2p)!}\cdot\frac{(6p)!(3p)}{(3p)!(6p-1)6p}\cdot\frac{p!}{(2p)!(2p+1)}
\end{eqnarray*}

This gives the formula in the statement.
\end{proof}

\section{Random matrices}

In this section we discuss two random matrix models for the measures $\pi_{st}$ with $s\in\mathbb N$. We restrict attention to the case $t=1$, since $\pi_{st}=\pi^{\boxtimes s-1}\boxtimes\pi^{\boxplus t}$ and therefore matrix models for $\pi_{st}$ will follow from matrix models for $\pi^{\boxtimes s}$.

We first recall the definition of a Wishart matrix.

Let $Y_1,Y_2,\dots,Y_p$ be independent vectors in $\mathbb{C}^N$ with identical Gaussian distribution $N(0,\Sigma)$ and set $W=Y_1Y_1^*+\dots +Y_pY_p^*$. Then the $N\times N$ Hermitian matrix $W$ follows the complex  Wishart distribution $W(N,p,\Sigma)$. 

We also can write $G^*=(Y_1,\dots,Y_p)$, so that $G$ is a $p\times N$ matrix and $W=G^{\ast} G$. When $\Sigma=\sigma^2 I_{N^2}$, then $G=(g_{ij})_{i=1\dots p,j=1\dots N}$ is a Gaussian
random matrix with independent entries of variance $\sigma^2$, that is such that $\{Re(g_{ij}),Im(g_{ij})\}$ is a family of $2pN$
independent $N(0,\sigma^2/2)$ random variables.

When $\Sigma=I_{N^2}/N$ and $\lim_{N\rightarrow
\infty}p/N=t$, the limiting spectral distribution of  $W$ is the free Poisson law of parameter $t$, i.e. is the measure $\pi_{1t}=\pi^{\boxplus t}$. See Haagerup and Thorbj{\o}rnsen \cite{ht}.

\begin{theorem}
Let $G_1,\ldots,G_s$ be a family of $N\times N$ independent  matrices formed by independent centered Gaussian variables, of variance $1/N$. Then with $M=G_1\ldots G_s$ the moments of the spectral distribution of $(MM^*)$ converge to the corresponding moments of $\pi_{s1}$, as $N\to\infty$.
\end{theorem} 

\begin{proof}
 We proceed by induction. At $s=1$ it is well-known that $MM^*$ is a model for $\pi_{11}$. So, assume that the result holds for $s-1\geq 1$. We have:
\begin{eqnarray*}
{\rm tr}(MM^*)^k
&=&{\rm tr}(G_1\ldots G_sG_s^*\ldots G_1^*)^k\\
&=&{\rm tr}(G_1(G_2\ldots G_sG_s^*\ldots G_1^*G_1)^{k-1}G_2\ldots G_sG_s^*\ldots G_1^*)
\end{eqnarray*}

We can pass the first $G_1$ matrix to the right, we get:
\begin{eqnarray*}
{\rm tr}(MM^*)^k
&=&{\rm tr}((G_2\ldots G_sG_s^*\ldots G_1^*G_1)^{k-1}G_2\ldots G_sG_s^*\ldots G_1^*G_1)\\
&=&{\rm tr}(G_2\ldots G_sG_s^*\ldots G_1^*G_1)^k\\
&=&{\rm tr}((G_2\ldots G_sG_s^*\ldots G_2^*)(G_1^*G_1))^k
\end{eqnarray*}

We know that $G_1^*G_1$ is a Wishart matrix, hence is a model for $\pi$.  Also, we know by the induction assumption that $G_2\ldots G_sG_s^*\ldots G_2^*$ gives a matrix model for $\pi_{s-1,1}$. Since by \cite{hp}, $G_1^*G_1$ and  $G_2\ldots G_sG_s^*\ldots G_2^*$ are asymptotically free, their product gives a matrix model for $\pi_{s-1,1}\boxtimes\pi_{11}=\pi_{s1}$, and we are done.
\end{proof}

\begin{theorem}
If $W$ is a $W(sN,sN,\frac{1}{sN}I_{(sN)^2})$ complex Wishart matrix and
$$D=\begin{pmatrix}
1_N&0&&0\cr
0&w1_N&&0\cr
&&\ddots&\cr
0&0&&w^{s-1}1_N
\end{pmatrix}$$
with $w=e^{2\pi i/s}$ then the mean empirical distribution of the eigenvalues of $(DW)^s$ converges to $\pi_{s1}$, as $N\to\infty$.
\end{theorem}

\begin{proof}
We use the formula of Graczyk, Letac and Massam \cite{glm}:
$$E({\rm Tr}(DW)^K)=\sum_{\sigma\in S_K}\frac{M^{\gamma(\sigma^{-1}\pi)}}{M^K}\,r_\sigma(D)$$

Here $W$ is a $W(M,M,\frac{1}{M}I_{N^2})$ complex Wishart matrix and $D$ is a deterministic $M\times M$. As for the right term, this is as follows:
\begin{enumerate}
\item $\pi$ is the cycle $(1,\ldots,K)$.
\item $\gamma(\sigma)$ is the number of disjoint cycles of $\sigma$.
\item If we denote by $C(\sigma)$ the set of such cycles and for any cycle $c$, by $|c|$ its length, then:
$$r_\sigma(D)=\prod_{c\in C(\sigma)}{\rm Tr}(D^{|c|})$$
\end{enumerate}

In our situation we have $K=sk$ and $M=sN$, and we get:
$$E({\rm Tr}(DW)^{sk})=
\sum_{\sigma\in S_{sk}}\frac{(sN)^{\gamma(\sigma^{-1}\pi)}}{(sN)^{sk}}\,r_\sigma(D)$$

Now since $D$ is uniformly formed by $s$-roots of unity, we have:
$${\rm Tr}(D^p)=
\begin{cases}
sN\mbox{ if }s|p\\
0\ \ \,\mbox{ if }s\!\!\not|p
\end{cases}$$

Thus if we denote by $S_{sk}^s$  the set of permutations $\sigma\in S_{sk}$ having the property that all the cycles of $\sigma$ have length multiple of $s$, the above formula reads:
$$E({\rm Tr}(DW)^{sk})=\sum_{\sigma\in S_{sk}^s}\frac{(sN)^{\gamma(\sigma^{-1}\pi)}}{(sN)^{sk}}\,(sN)^{\gamma(\sigma)}$$

In terms of the normalized trace tr, we get:
$$E({\rm tr}(DW)^{sk})=\sum_{\sigma\in S_{sk}^s}(sN)^{\gamma(\sigma^{-1}\pi)+\gamma(\sigma)-sk-1}$$

The exponent on the right, say $L_\sigma$, can be estimated by using the distance on the Cayley graph of $S_{sk}$:
\begin{eqnarray*}
L_\sigma
&=&\gamma(\sigma^{-1}\pi)+\gamma(\sigma)-sk-1\\
&=&(sk-d(\sigma,\pi))+(sk-d(e,\sigma))-sk-1\\
&=&sk-1-(d(e,\sigma)+d(\sigma,\pi))\\
&\leq&sk-1-d(e,\pi)\\
&=&0
\end{eqnarray*}

Now when taking the limit $N\to\infty$ in the above formula of $E({\rm tr}(DW)^{sk})$, the only terms that count are those coming from permutations $\sigma\in S_{sk}^s$ having the property $L_\sigma=0$, which each contribute with a 1 value. We get:
\begin{eqnarray*}
\lim_{N\to\infty}E({\rm tr}(DW)^{sk})
&=&\#\{\sigma\in S_{sk}^s\ |\ L_\sigma=0\}\\
&=&\#\{\sigma\in S_{sk}^s\ |\ d(e,\sigma)+d(\sigma,\pi)=d(e,\pi)\}\\
&=&\#\{\sigma\in S_{sk}^s\ |\ \sigma\in [e,\pi]\}
\end{eqnarray*}

Now by using Biane's correspondence in \cite{bia}, this is the same as the number of noncrossing partitions of $\{1,\ldots,sk\}$ having all blocks of size multiple of $s$. Thus we have reached to the sets $NC_s(k)$ from section 4, and we are done.
\end{proof}

As a consequence of the above random matrix formula, we have the following alternative free probabilistic approach to the free Bessel laws.

\begin{theorem}
The free Bessel law $\pi_{s1}$ with $s\in\mathbb N$ is given by
$$\pi_{s1}={\rm law}\left(\sum_{k=1}^sw^k\alpha_k\right)^s$$
where $\alpha_1,\ldots,\alpha_s$ are free random variables, each of them following the free Poisson law of parameter $1/s$, and $w=e^{2\pi i/s}$.
\end{theorem}

\begin{proof}
Let $G_1,\ldots,G_s$ be a family of independent $sN\times N$ matrices formed by independent, centered complex Gaussian variables, of variance $1/(sN)$. The following matrices $H_1,\ldots,H_s$ are as well complex Gaussian and independent:
$$H_k=\frac{1}{\sqrt{s}}\sum_{p=1}^sw^{kp}G_p$$

Thus the following matrix provides a model for $\Sigma w^k\alpha_k$:
\begin{eqnarray*}
M
&=&\sum_{k=1}^sw^kH_kH_k^*\\
&=&\frac{1}{s}\sum_{k=1}^s\sum_{p=1}^s\sum_{q=1}^sw^{k+kp-kq}G_pG_q^*\\
&=&\sum_{p=1}^s\sum_{q=1}^s\left(\frac{1}{s}\sum_{k=1}^s\left(w^{1+p-q}\right)^k\right)G_pG_q^*\\
&=&G_1G_2^*+G_2G_3^*+\ldots+G_{s-1}G_s^*+G_sG_1^*
\end{eqnarray*}

This matrix can be written as:
\begin{eqnarray*}
M
&=&\begin{pmatrix}G_1&G_2&\ldots&G_{s-1}&G_s\end{pmatrix}
\begin{pmatrix}G_2^*\cr G_3^*\cr\ldots\cr G_s^*\cr G_1^*\end{pmatrix}\\
&=&\begin{pmatrix}G_1&G_2&\ldots&G_{s-1}&G_s\end{pmatrix}
\begin{pmatrix}
0&1_N&0&\ldots&0\cr
0&0&1_N&\ldots&0\cr
&&&\ddots&&\cr
0&0&0&\ldots&1_N\cr
1_N&0&0&\ldots&0
\end{pmatrix}
\begin{pmatrix}G_1^*\cr G_2^*\cr\ldots\cr G_{s-1}^*\cr G_s^*\end{pmatrix}\\
&=&GOG^*
\end{eqnarray*}

Here $G=(G_1\ \ldots\  G_s)$ is the $sN\times sN$ Gaussian matrix obtained by concatenating $G_1,\ldots,G_s$, and $O$ is the matrix in the middle. But this latter matrix is of the form $O=UDU^*$ with $U$ unitary, so we have $M=GUDU^*G^*$. Now since $GU$ is a Gaussian matrix, $M$ has the same law as $M'=GDG^*$, and we get:
\begin{eqnarray*}
E\left(\left(\sum_{l=1}^sw^l\alpha_l\right)^{sk}\right)
&=&\lim_{N \rightarrow + \infty}  E({\rm tr}(M^{sk}))\\
&=&\lim_{N \rightarrow + \infty} E({\rm tr}(GDG^*)^{sk})\\
&=&\lim_{N \rightarrow + \infty} E({\rm tr}(D(G^*G))^{sk})
\end{eqnarray*}

Thus with $W=G^*G$ we get the result.
\end{proof}

\section{Free additivity}

In this section we investigate the free additivity property of $\pi_{st}$, in analogy with the well-known free additivity property of the free Poisson laws $\pi_{1t}$.

We begin with a generalization of Theorem 6.3.

\begin{theorem}
The free Bessel law $\pi_{st}$ with $s\in\mathbb N$ is given by
$$\pi_{st}={\rm law}\left(\sum_{k=1}^sw^k\alpha_k\right)^s$$
where $\alpha_1,\ldots,\alpha_s$ are free random variables, each of them following the free Poisson law of parameter $t/s$, and $w=e^{2\pi i/s}$.
\end{theorem}

\begin{proof}
Given a random variable $\alpha$ and a complex number $q$, we have the following relations between the functional transforms of ${\rm law}(\alpha)$ and ${\rm law}(q\alpha)$:
\begin{eqnarray*}
f_{q\alpha}(z)=f_\alpha(qz)
&\implies&G_{q\alpha}(z)=q^{-1}G_\alpha(q^{-1}z)\\
&\implies&K_{q\alpha}(z)=qK_\alpha(qz)\\
&\implies&R_{q\alpha}(z)=qR_\alpha(qz)
\end{eqnarray*}

Consider now the variable $\alpha=\sum w^k\alpha_k$. We have:
$$R_\alpha(z)
=\sum_{k=1}^sw^kR_{\alpha_k}(w^kz)
=\sum_{k=1}^sw^k\cdot\frac{t}{s}\cdot\frac{1}{1-w^kz}$$

This gives the following formula:
$$R_\alpha(z)
=t\left(\frac{1}{s}\sum_{k=1}^s\frac{w^k}{1-w^kz}\right)
=\frac{tz^{s-1}}{1-z^s}$$

Consider now the formal measure $\tilde{\pi}_{st}$ having Stieltjes transform $\tilde{f}(z)=f(z^s)$, where $f$ is the Stieltjes transform of $\pi_{st}$. The $R$ transform of $\tilde{\pi}_{st}$ can be computed by using the equation of $f$ in Proposition 4.1:
\begin{eqnarray*}
f=1+zf^s(f+t-1)
&\implies&\tilde{f}=1+(z\tilde{f})^s(\tilde{f}+t-1)\\
&\implies&\xi\tilde{G}=1+\tilde{G}^s(\xi \tilde{G}+t-1)\\
&\implies&\tilde{K}z=1+z^s(\tilde{K}z+t-1)\\
&\implies&\tilde{R}z+1=1+z^s(\tilde{R}z+t)\\
&\implies&\tilde{R}(z)=tz^{s-1}/(1-z^s)
\end{eqnarray*}

Thus we have the equality of $R$ transforms $\tilde{R}=R_\alpha$. In terms of measures we get $\tilde{\pi}_{st}={\rm law}(\alpha)$, hence $\pi_{st}={\rm law}(\alpha^s)$, and we are done.
\end{proof}

It is convenient to introduce the following measures.

\begin{definition}
The modified free Bessel laws $\tilde{\pi}_{st}$ with $s\in\mathbb N$ are given by
$$\tilde{\pi}_{st}={\rm law}\left(\sum_{k=1}^sw^k\alpha_k\right)$$
where $\alpha_1,\ldots,\alpha_s$ are free random variables, each of them following the free Poisson law of parameter $t/s$, and $w=e^{2\pi i/s}$.
\end{definition}

We know from the previous section that the mean empirical distribution of the eigenvalues of $DW$ converges towards $\tilde{\pi}_{s1}$, with the notations there. Also, we know from the previous proof that the family $\tilde{\pi}_{st}$ is freely additive with respect to $t$, the $R$ transform being given by $\tilde{R}_{st}=tz^{s-1}/(1-z^s)$.

These results show that we have $\tilde{\pi}_{st}=\pi_{t\rho}$, where $\rho$ is the uniform measure on the $s$-roots of unity, and $\pi_{t\rho}$ is the corresponding compound free Poisson law. 

For real measures $\rho$, the compound free Poisson laws $\pi_{t\rho}$ were introduced by Speicher in \cite{sp}, and studied by Hiai and Petz in \cite{hp}. In our case $\rho$ is complex, but the main results (R-transform, matrix models) still hold. As for the third main result, this is the Poisson limit one, which in our case is as follows.

\begin{theorem}
We have the Poisson limit type convergence
$$\left(\left(1-\frac{1}{n}\right)\delta_0+\frac{1}{n}\,\rho\right)^{\boxplus n}\to\tilde{\pi}_{s1}$$
where $\rho$ is the uniform measure on the $s$-roots of unity.
\end{theorem}

\begin{proof}
We compute first the $R$ transform of the measure on the left:
\begin{eqnarray*}
\mu=\left(1-\frac{1}{n}\right)\delta_0+\frac{1}{n}\,\rho
&\implies&f=\left(1-\frac{1}{n}\right)+\frac{1}{n}\cdot\frac{1}{1-z^s}\\
&\implies&G(\xi)=\frac{1}{\xi}+\frac{1}{n}\cdot\frac{1}{\xi(\xi^s-1)}\\
&\implies&(K^s-1)(zK-1)=\frac{1}{n}\\
&\implies&\left(\left(R+\frac{1}{z}\right)^s-1\right)zR=\frac{1}{n}
\end{eqnarray*}

This shows that the $R$ transform of $\mu^{\boxplus n}$ satisfies:
$$\left(\left(\frac{R}{n}+\frac{1}{z}\right)^s-1\right)z\,\frac{R}{n}=\frac{1}{n}$$

We multiply by $n$, then we take the limit $n\to\infty$. We get:
$$\left(\frac{1}{z^s}-1\right)zR=1$$

Thus in the limit $n\to\infty$ we have $R=z^{s-1}/(1-z^s)$, and we are done.
\end{proof}

\section{Classical analogues}

We discuss now the classical analogues $p_{st},\tilde{p}_{st}$ of the free Bessel laws $\pi_{st},\tilde{\pi}_{st}$. There are several ways of introducing these laws. Most convenient is to start with the following formulae, similar to those in Theorem 7.1 and Definition 7.2.

\begin{definition}
The Bessel laws $p_{st}$ and the modified Bessel laws $\tilde{p}_{st}$ with $s\in\mathbb N$ are given by
$$p_{st}={\rm law}\left(\sum_{k=1}^sw^ka_k\right)^s$$
$$\tilde{p}_{st}={\rm law}\left(\sum_{k=1}^sw^ka_k\right)$$
where $a_1,\ldots,a_s$ are independent random variables, each of them following the Poisson law of parameter $t/s$, and $w=e^{2\pi i/s}$.
\end{definition}
 
As a first remark, at $s=1$ we get the Poisson law of parameter $t$:
$$p_{1t}=\tilde{p}_{1t}=e^{-t}\sum_{r=0}^\infty \frac{t^r}{r!}\,\delta_r$$

In what follows we present a number of results, which show that $p_{st},\tilde{p}_{st}$ are indeed the classical analogues of $\pi_{st},\tilde{\pi}_{st}$, for any $s\in\mathbb N$.  

The first such interpretation comes from the Bercovici-Pata bijection \cite{bp}. Since this makes correspond Poisson laws to free Poisson laws, and convolution to free convolution, we get by linearity that it makes correspond $\tilde{p}_{st}$ and $\tilde{\pi}_{st}$. 

We discuss now the additivity property and the Poisson limit convergence for Bessel laws, in analogy with the considerations from the previous section.

We use the level $s$ exponential function:
$$\exp_sz=\sum_{k=0}^\infty\frac{z^{sk}}{(sk)!}$$

We have the following formula, in terms of $w=e^{2\pi i/s}$:
$$\exp_sz=\frac{1}{s}\sum_{k=1}^s\exp(w^kz)$$

Observe that we have $\exp_1=\exp$ and $\exp_2=\cosh$.

\begin{theorem}
The Fourier transform of $\tilde{p}_{st}$ is given by
$$\log\tilde{F}_{st}(z)=t\left(\exp_sz-1\right)$$
so in particular the measures $\tilde{p}_{st}$ are additive with respect to $t$.
\end{theorem}

\begin{proof}
Consider the variable $a=\sum w^ka_k$. For the Poisson law of parameter $t$ we have $\log F(z)=t(e^z-1)$, and by using the identity $F_{qa}(z)=F_a(qz)$, we get:
$$\log F_a(z)
=\sum_{k=1}^s\log F_{a_k}(w^kz)
=\sum_{k=1}^s\frac{t}{s}\left(\exp(w^kz)-1\right)$$

This gives the following formula:
$$\log F_a(z)
=t\left(\left(\frac{1}{s}\sum_{k=1}^s\exp(w^kz)\right)-1\right)
=t\left(\exp_s(z)-1\right)$$

Now since $\tilde{p}_{st}$ is the law of $a$, this gives the formula in the statement.
\end{proof}
  
\begin{theorem}
We have the Poisson limit type convergence
$$\left(\left(1-\frac{1}{n}\right)\delta_0+\frac{1}{n}\,\rho\right)^{*n}\to\tilde{p}_{s1}$$
where $\rho$ is the uniform measure on the $s$-roots of unity.
\end{theorem}

\begin{proof}
We compute first the Fourier transform of the measure on the left:
$$\mu=\left(1-\frac{1}{n}\right)\delta_0+\frac{1}{n}\,\rho
\implies F=\left(1-\frac{1}{n}\right)+\frac{1}{n}\exp_s(z)$$

This shows that the Fourier transform of $\mu^{*n}$ is given by:
\begin{eqnarray*}
F
&=&\left(\left(1-\frac{1}{n}\right)+\frac{1}{n}\exp_s(z)\right)^n\\
&=&\left(1+\frac{\exp_s(z)-1}{n}\right)^n\\
&\simeq&\exp(\exp_s(z)-1)
\end{eqnarray*}

Thus in the limit $n\to\infty$ we have $\log F=\exp_sz-1$, and we are done.
\end{proof}

\section{Bessel functions}

In this section we study the densities of $p_{st},\tilde{p}_{st}$. At $s=2$ this will lead to Bessel functions, which will justify the general terminology for $\pi_{st},\tilde{\pi}_{st}$.

\begin{theorem}
We have the formulae
$$p_{st}=e^{-t}\sum_{p_1=0}^\infty\ldots\sum_{p_s=0}^\infty\frac{1}{p_1!\ldots p_s!}\,\left(\frac{t}{s}\right)^{p_1+\ldots+p_s}\delta\left(\sum_{k=1}^sw^kp_k\right)^s$$
$$\tilde{p}_{st}=e^{-t}\sum_{p_1=0}^\infty\ldots\sum_{p_s=0}^\infty\frac{1}{p_1!\ldots p_s!}\,\left(\frac{t}{s}\right)^{p_1+\ldots+p_s}\delta\left(\sum_{k=1}^sw^kp_k\right)$$
where $w=e^{2\pi i/s}$, and the $\delta$ symbol is a Dirac mass.
\end{theorem}

\begin{proof}
It is enough to prove the formula for $\tilde{p}_{st}$. For this purpose, we compute the Fourier transform of the measure on the right. This is given by:
\begin{eqnarray*}
F(z)
&=&e^{-t}\sum_{p_1=0}^\infty\ldots\sum_{p_s=0}^\infty\frac{1}{p_1!\ldots p_s!}\left(\frac{t}{s}\right)^{p_1+\ldots+p_s}F\delta\left(\sum_{k=1}^sw^kp_k\right)(z)\\
&=&e^{-t}\sum_{p_1=0}^\infty\ldots\sum_{p_s=0}^\infty\frac{1}{p_1!\ldots p_s!}\left(\frac{t}{s}\right)^{p_1+\ldots+p_s}\exp\left(\sum_{k=1}^sw^kp_kz\right)\\
&=&e^{-t}\sum_{r=0}^\infty\left(\frac{t}{s}\right)^r\sum_{\Sigma p_i=r}\frac{\exp\left(\sum_{k=1}^sw^kp_kz\right)}{p_1!\ldots p_s!}
\end{eqnarray*}

We multiply by $e^t$, and we compute the derivative with respect to $t$:
\begin{eqnarray*}
(e^tF(z))'
&=&\sum_{r=1}^\infty\frac{r}{s}\left(\frac{t}{s}\right)^{r-1}\sum_{\Sigma p_i=r}\frac{\exp\left(\sum_{k=1}^sw^kp_kz\right)}{p_1!\ldots p_s!}\\
&=&\frac{1}{s}\sum_{r=1}^\infty\left(\frac{t}{s}\right)^{r-1}\sum_{\Sigma p_i=r}\left(\sum_{l=1}^sp_l\right)\frac{\exp\left(\sum_{k=1}^sw^kp_kz\right)}{p_1!\ldots p_s!}\\
&=&\frac{1}{s}\sum_{r=1}^\infty\left(\frac{t}{s}\right)^{r-1}\sum_{\Sigma p_i=r}\sum_{l=1}^s\frac{\exp\left(\sum_{k=1}^sw^kp_kz\right)}{p_1!\ldots p_{l-1}!(p_l-1)!p_{l+1}!\ldots p_s!}\\
\end{eqnarray*}

By using the variable $u=r-1$, we get:
\begin{eqnarray*}
(e^tF(z))'
&=&\frac{1}{s}\sum_{u=0}^\infty\left(\frac{t}{s}\right)^u\sum_{\Sigma q_i=u}\sum_{l=1}^s\frac{\exp\left(w^lz+\sum_{k=1}^sw^kq_kz\right)}{q_1!\ldots q_s!}\\
&=&\left(\frac{1}{s}\sum_{l=1}^s\exp(w^lz)\right)\left(\sum_{u=0}^\infty\left(\frac{t}{s}\right)^u\sum_{\Sigma q_i=u}\frac{\exp\left(\sum_{k=1}^sw^kq_kz\right)}{q_1!\ldots q_s!}\right)\\
&=&(\exp_sz)(e^t\tilde{F}_{st}(z))
\end{eqnarray*}

On the other hand, the function $\Phi(t)=\exp(t\exp_sz)$ satisfies as well the equation $\Phi'(t)=(\exp_sz)\Phi(t)$. Thus we have $e^tF(z)=\Phi(t)$, which gives:
\begin{eqnarray*}
\log F
&=&\log(e^{-t}\exp(t\exp_sz))\\
&=&\log(\exp(t(\exp_sz-1)))\\
&=&t(\exp_sz-1)
\end{eqnarray*}

This gives the formulae in the statement.
\end{proof}

Recall now that the Bessel function of the first kind is given by:
$$\varphi_r(t)=\sum_{p=0}^\infty \frac{t^{2p+r}}{p!(p+r)!}$$

The following result justifies the terminology used in this paper.

\begin{theorem}
We have the formulae
$$p_{2t}=e^{-t}\sum_{r=-\infty}^\infty\varphi_{|r|}\left(\frac{t}{2}\right)\delta_{r^2}$$
$$\tilde{p}_{2t}=e^{-t}\sum_{r=-\infty}^\infty\varphi_{|r|}\left(\frac{t}{2}\right)\delta_r$$
where $\varphi_r$ is the Bessel function of the first kind.
\end{theorem}

\begin{proof}
At $s=2$ the primitive root of unity is $w=-1$, and we get: 
\begin{eqnarray*}
\tilde{p}_{2t}
&=&e^{-t}\sum_{p=0}^\infty\sum_{q=0}^\infty\frac{(t/2)^{p+q}}{p!q!}\,\delta_{p-q}\\
&=&e^{-t}\sum_{r=-\infty}^\infty\sum_{p-q=r}\frac{(t/2)^{p+q}}{p!q!}\,\delta_r\\
&=&e^{-t}\left(\sum_{r=0}^\infty\sum_{q=0}^\infty\frac{(t/2)^{r+2q}}{(r+q)!q!}\delta_r
+\sum_{r=-\infty}^{-1}\sum_{p=0}^\infty\frac{(t/2)^{2p-r}}{p!(p-r)!}\delta_r\right)\\
\end{eqnarray*}

Thus the density of $\tilde{p}_{2t}$ is given indeed by the Bessel function:
\begin{eqnarray*}
\tilde{p}_{2t}
&=&e^{-t}\left(\sum_{r=0}^\infty\sum_{q=0}^\infty\frac{(t/2)^{r+2q}}{(r+q)!q!}\delta_r
+\sum_{r=-\infty}^{-1}\sum_{p=0}^\infty\frac{(t/2)^{2p+|r|}}{p!(p+|r|)!}\delta_r\right)\\
&=&e^{-t}\sum_{r=-\infty}^\infty\sum_{p=0}^\infty\frac{(t/2)^{|r|+2p}}{(|r|+p)!p!}\,\delta_r
\end{eqnarray*}

This gives the formulae in the statement.
\end{proof}

We know that $p_{1t},p_{2t}$ are supported by $\mathbb N,\mathbb Z$. In the general case the situation is a bit more complicated: the support is formed by the $s$ powers of certain elements in $\mathbb Z[w]$, so we can only say that it is contained in $\mathbb Z [w]$. As for the density, this should be thought of as being a kind of $s$-dimensional Bessel function.

\section{Quantum groups}

We discuss now the representation theory approach to $p_{st},\pi_{st}$. The material, presented in this section and in the next two ones, is organized as follows:
\begin{enumerate}
\item We first discuss the case $s=1,2$, by surveying some previously known results, from \cite{bjb}, \cite{bc2}, \cite{bbc}. In order to simplify the presentation, we use in this introductory part the formalism of compact quantum groups.
\item Then we discuss the case of arbitrary $s\in\mathbb N$, with all the definitions and results written by using Woronowicz's Hopf algebra formalism in \cite{w1}.
\end{enumerate}

Consider a compact group $G\subset U_n$. The character of the fundamental representation $\chi:G\to\mathbb C$ is by definition the restriction to $G$ of the usual trace:
$$\chi(g)=\sum_{i=1}^ng_{ii}$$

In functional analytic terms, the character $\chi\in C(G)$ can be defined starting with the $n^2$ matrix coordinates $u_{ij}\in C(G)$, as being the trace of $u=(u_{ij})$: 
$$\chi=\sum_{i=1}^nu_{ii}$$

The law of $\chi$ with respect to the Haar functional of $C(G)$ is a fundamental object in representation theory, because of the following formula:
$$\int\chi^k=\#\{1\in u^{\otimes k}\}$$

Here we regard $C(G)$ as a Hopf $C^*$-algebra, and the number of the right is the multiplicity of the trivial corepresentation $1$ into the $k$-th tensor power $u^{\otimes k}=u_{1,k+1}u_{2,k+1}\ldots u_{k,k+1}$. This corepresentation has character $\chi^k$, and the above formula comes from the well-known fact that the number of copies of $1$ can be obtained by integrating the character. See e.g. Woronowicz \cite{w1}.

The following statement from \cite{bjb} is a representation theory interpretation of the relationship between the Poisson law $p$ and the free Poisson law $\pi$.

\begin{proposition}
We have the following formulae.
\begin{enumerate}
\item For $G=S_n$ with $n\to\infty$ we have ${\rm law}(\chi)\to p$.
\item For $G=S_n^+$ with $n\geq 4$ we have ${\rm law}(\chi)=\pi$.
\end{enumerate}
\end{proposition}

Here $S_n^+$ is Wang's quantum permutation group \cite{wa}. This quantum group doesn't exist as a concrete object, but the character under investigation exists, as an element of the associated Hopf algebra $A_s(n)$. See \cite{bjb} for details.

Observe that there is a slight problem with the above statement, which is not fully symmetric in terms of convergences. As pointed out in \cite{bc2}, a uniform statement can be obtained in terms of truncated characters, given by:
$$\chi_t=\sum_{i=1}^{[tn]}u_{ii}$$

Here $t\in (0,1]$ is a parameter.

The laws of truncated characters can be computed by using the Weingarten formula, and we have the following result \cite{bc2}.

\begin{theorem}
We have the following formulae.
\begin{enumerate}
\item For $G=S_n$ with $n\to\infty$ we have ${\rm law}(\chi_t)\to p_{1t}$.
\item For $G=S_n^+$ with $n\to\infty$ we have ${\rm law}(\chi_t)\to\pi_{1t}$.
\end{enumerate}
\end{theorem}

The second result of this type concerns the hyperoctahedral group $H_n$. This is the symmetry group of the cube in $\mathbb R^n$. 

The hyperoctahedral group has a wreath product decomposition $H_n=\mathbb Z_2\wr S_n$, and its free version $H_n^+$ has a free wreath product decomposition $H_n^+=\mathbb Z_2\wr_*S_n^+$.

The laws of truncated characters can be computed by using wreath product techniques and the Weingarten formula, and we have the following result \cite{bbc}.

\begin{theorem}
We have the following formulae.
\begin{enumerate}
\item For $G=H_n$ with $n\to\infty$ we have ${\rm law}(\chi_t)\to\tilde{p}_{2t}$.
\item For $G=H_n^+$ with $n\to\infty$ we have ${\rm law}(\chi_t)\to\tilde{\pi}_{2t}$.
\end{enumerate}
\end{theorem}

Summarizing, the groups $S_n,H_n$ and their free analogues $S_n^+,H_n^+$ provide models for the Bessel and free Bessel laws $\tilde{p}_{st},\tilde{\pi}_{st}$, at $s=1,2$. 

In what follows we will generalize these results, with a single two-fold statement (classical and quantum) which works for any $s\in\mathbb N$. 

Together with the additional results in \cite{bc1}, concerning the groups $O_n,U_n$ and their free versions $O_n^+,U_n^+$, this will justify the table in the introduction.

\section{The group $H_n^s$}

We discuss here the generalization of (1) in Theorem 10.2 and Theorem 10.3. We discuss as well the generalization of (2), with some preliminary facts.

A matrix is called monomial if it has exactly one nonzero entry in each row and each column. The basic examples are the permutation matrices.

\begin{definition}
$H_n^s=\mathbb Z_s\wr S_n$ is the group of monomial $n\times n$ matrices having as entries the $s$-roots of unity. 
\end{definition}

In other words, an element of $H_n^s$ is a permutation matrix, with each $1$ entry replaced by a $s$-root of unity. When identifying the group of $s$-roots of unity with $\mathbb Z_s$, this gives the wreath product decomposition in the above definition.

Observe that we have $H_n^s\subset U_n$, and that $H_n^1=S_n,H_n^2=H_n$.

\begin{theorem}
For $H_n^s$ with $n\to\infty$ we have ${\rm law}(\chi_t)\to\tilde{p}_{st}$.
\end{theorem}

\begin{proof}
For $s=1,2$ this is known from \cite{bc2}, \cite{bbc}, and we will use here the same method. We denote by $\rho$ the uniform measure on the $s$-roots of unity.

We work out first the case $t=1$. Since the probability for a random permutation to have exactly $k$ fixed points is $e^{-1}/k!$, we get:
$$\lim_{n\to\infty}{\rm law}(\chi_1)=e^{-1}\sum_{k=0}^\infty \frac{1}{k!}\,\rho^{*k}$$

On the other hand, we get from Theorem 8.3:
\begin{eqnarray*}
\tilde{p}_{s1}
&=&\lim_{n\to\infty}\left(\left(1-\frac{1}{n}\right)\delta_0+\frac{1}{n}\,\rho\right)^{*n}\\
&=&\lim_{n\to\infty}\sum_{k=0}^n\begin{pmatrix}n\cr k\end{pmatrix}\left(1-\frac{1}{n}\right)^{n-k}\frac{1}{n^k\,}\rho^{*k}\\
&=&e^{-1}\sum_{k=0}^\infty \frac{1}{k!}\,\rho^{*k}
\end{eqnarray*}

This gives the assertion for $t=1$. Now in the case $t>0$ arbitrary, we can use the same method, by performing the following modifications:
\begin{eqnarray*}
\lim_{n\to\infty}{\rm law}(\chi_t)
&=&e^{-t}\sum_{k=0}^\infty \frac{t^k}{k!}\,\rho^{*k}\\
&=&\lim_{n\to\infty}\left(\left(1-\frac{1}{n}\right)\delta_0+\frac{1}{n}\,\rho\right)^{*[tn]}\\
&=&\tilde{p}_{st}
\end{eqnarray*}

This finishes the proof.
\end{proof}

It remains to discuss the quantum group model for $\pi_{st}$. The quantum group will be a suitably chosen free version of $H_n^s$. 

\begin{definition}
The universal $C^*$-algebra $A_h^s(n)$ is defined with $n^2$ generators $u_{ij}$, and with the following relations:
\begin{enumerate}
\item $u=(u_{ij})$ and $\bar{u}=(u_{ij}^*)$ are unitaries.
\item $u_{ij}u_{ij}^*=u_{ij}^*u_{ij}=p_{ij}$ (projection). 
\item $u_{ij}^s=p_{ij}$.
\end{enumerate}
\end{definition}

In this definition the meaning of the second condition is that each $u_{ij}$ is normal, and that the elements $p_{ij}=u_{ij}u_{ij}^*$ are idempotents: $p_{ij}^2=p_{ij}$. Observe that each $u_{ij}$ is a normal partial isometry, in the $C^*$-algebra sense. 

We use the symmetric and hyperoctahedral Hopf algebras $A_s(n),A_h(n)$, associated to the free quantum groups $S_n^+,H_n^+$. See \cite{bc2}, \cite{bbc}.

\begin{proposition}
The algebra $A_h^s(n)$ has the following properties.
\begin{enumerate}
\item It is a Hopf $C^*$-algebra, with $u$ being a corepresentation.
\item Its maximal commutative quotient is $C(H_n^s)$.
\item $A_h^1(n)=A_s(n)$, $A_h^2(n)=A_h(n)$.
\end{enumerate}
\end{proposition}

\begin{proof}
We use the notation $A=A_h^s(n)$.

(1) The comultiplication, counit and antipode can be constructed by using the universal property of $A$, according to the following formulae:
\begin{eqnarray*}
\Delta(u_{ij})&=&\sum_{k=1}^nu_{ik}\otimes u_{kj}\\
\varepsilon(u_{ij})&=&\delta_{ij}\\
S(u_{ij})&=&u_{ji}^*
\end{eqnarray*}

Indeed, the matrices $(\Delta u)_{ij}=\Sigma_ku_{ik}\otimes u_{kj}$, $(\varepsilon u)_{ij}=\delta_{ij}$, $(Su)_{ij}=u_{ji}^*$ satisfy the conditions in Definition 11.3, so we can define $\Delta,\varepsilon,S$ as above. This shows that we have a Hopf $C^*$-algebra in the sense of Woronowicz \cite{w1}, and we are done.

(2) The matrix coordinates of $H_n^s\subset U_n$ satisfy the relations in Definition 11.3, so we have a surjective morphism $A\to C(H_n^s)$. 

Consider the ideal $J\subset A$ generated by the relations $[u_{ij},u_{kl}]=0$. It is routine to check that $A/J$ is a Hopf algebra, with $\Delta,\varepsilon,S$ defined as above. Thus we have $A/J=C(G)$, where $G\subset U_n$ is a certain compact group, containing $H_n^s$. 

From the fact that $u$ is unitary we get that the matrix of projections $p=(p_{ij})$ has sum 1 on each row and each column. It follows that the entries $p_{ij}$ are pairwise orthogonal on rows and columns, and we get that $G$ consists of monomial matrices. Now from the relation $u_{ij}^s=p_{ij}$ we get $G\subset H_n^s$, and we are done. 

(3) This follows from a routine comparison between Definition 11.3 and the definitions in \cite{bbc}, by using the above-mentioned fact that when a number of projections sum up to $1$, they are pairwise orthogonal.
\end{proof}

Summarizing, $A_h^s(n)$ appears to be the natural candidate for a model for $\pi_{st}$. We will prove in the next section that indeed it is so.

We should mention that $A_h^s(n)$ has a number of other interesting properties, not to be investigated here. One can prove for instance that we have a decomposition $A_h^s(n)=C(\mathbb Z_s)*_wA_s(n)$, analogous to the decomposition $H_n^s=\mathbb Z_s\wr S_n$. Here $*_w$ is a free wreath product in the sense of Bichon \cite{bic}, and the proof is as in \cite{bbc}, by using a suitable reformulation of Definition 11.3, as a sudoku type condition.

\section{Integration over $A_h^s(n)$}

We compute now the asymptotic laws of truncated characters for the algebra $A_h^s(n)$. The integration is with respect to the Haar functional, known to exist by general results of Woronowicz in \cite{w1}. 

\begin{theorem}
For $A_h^s(n)$ with $n\to\infty$ we have ${\rm law}(\chi_t)\to\tilde{\pi}_{st}$.
\end{theorem}

\begin{proof}
We use a standard method, developed in \cite{bc1}, \cite{bc2}, \cite{bbc}, \cite{ba}. The general idea is that at $s=1,2$ the result is known from \cite{bc2}, \cite{bbc}, so we can use an extension of the proof there. The main problem at $s\geq 3$ comes from the fact that the fundamental corepresentation $u$ is no longer self-adjoint, so the underlying combinatorial objects will be indexed by elements of $\mathbb N*\mathbb N$, rather than by numbers in $\mathbb N$. In order to deal with this problem, we use methods from \cite{bc1}, \cite{ba}.

The proof uses tensor categories, denoted $C$, and has 9 steps:
\begin{enumerate}
\item We introduce three auxiliary algebras: $A_k(n),A_h^\infty(n),A_s(n)$.
\item We discuss the relation between the associated categories.
\item We describe the passage $CA_k(n)\to CA_h^\infty(n)$.
\item We compute $CA_h^\infty(n)$.
\item We describe the passage $CA_h^\infty(n)\to CA_h^s(n)$.
\item We compute $CA_h^s(n)$.
\item We discuss the integration formula for characters.
\item We prove the result for $t=1$.
\item We prove the result for any $t>0$.
\end{enumerate}

As already mentioned, the method is quite standard, so we will insist on technical details only. In fact, the only problem comes from the overall level of complexity, which is higher than in the previous papers \cite{bc1}, \cite{bc2}, \cite{bbc}, \cite{ba}. Let us also mention that in the last part of the proof, the cumulant computations can be probably deduced as well from general results of Lehner in \cite{le}. 

{\bf Step 1.} Let $A_h^\infty(n)$ be the algebra defined as $A_h^s(n)$, but with the condition (3) in Definition 11.3 missing. That is, $A_h^\infty(n)$ is the algebra generated by $n^2$ normal partial isometries $u_{ij}$, such that $u=(u_{ij})$ and $\bar{u}=(u_{ij}^*)$ are unitaries.

Observe that we have an arrow as follows, for any $s\in\mathbb N$:
$$A_h^\infty(n)\to A_h^s(n)$$

By arguing like in the proof of Proposition 11.4 we get that $A_h^\infty(n)$ is a Hopf algebra, having as maximal commutative quotient the algebra of functions on the group $H_n^\infty=\mathbb T\wr S_n$, consisting of unitary monomial matrices.

Consider also the algebra $A_k(n)$ generated by $n^2$ variables $u_{ij}$, having the property that $u=(u_{ij})$ and $\bar{u}=(u_{ij}^*)$ are unitaries, and that the relation $ab^*=a^*b=0$ holds for $a,b$ distinct entries on the same row or column of $u$.

Once again by arguing like in the proof of Proposition 11.4, we get that $A_k(n)$ is a Hopf algebra, having $C(H_n^\infty)$ as maximal commutative quotient. See \cite{ba}.

Observe that we have an arrow as follows:
$$A_k(n)\to A_h^\infty(n)$$

This follows indeed from the orthogonality condition on the supporting projections $p_{ij}$, obtained in the proof of Proposition 11.4.

The fact that $A_k(n),A_h^\infty(n)$ have the same maximal commutative quotient might seem a bit surprising. The point is that when trying to liberate the commutative Hopf algebra $C(H_n^\infty)$, the normality condition of the elements $u_{ij}$ can be kept or not. This is why we end up with two different algebras.

Finally, consider the algebra $A_s(n)=A_h^1(n)$. This is Wang's quantum permutation algebra \cite{wa}, which can be also described as being the quotient of $A_k(n)$ by the relations making each $u_{ij}$ a projection.

Observe that we have an arrow as follows:
$$A_h^s(n)\to A_s(n)$$

Summarizing, the algebra $A_h^s(n)$ under investigation, and its $s=\infty$ version, are part of the following sequence:
$$A_k(n)\to A_h^\infty(n)\to A_h^s(n)\to A_s(n)$$

The point is that the algebras on the left and on the right are well understood, and this can be used for studying the algebras in the middle. 

{\bf Step 2.} We denote by $CA$ the tensor category associated to a pair $(A,u)$ as in Woronowicz's paper \cite{w1}. That is, the objects are the tensor products between $u$ and $\bar{u}$, and the arrows are the intertwiners between them. Observe that in the case $u=\bar{u}$, the objects are just the tensor powers of $u$. 

Since applying morphisms increases the spaces of intertwiners, we have embeddings of tensor categories as follows:
$$CA_k(n)\subset CA_h^\infty(n)\subset CA_h^s(n)\subset CA_s(n)$$

The result that we want to prove is of asymptotic nature, so we can make the assumption $n\geq 4$. Now with this condition in hand, it is well-known that $CA_s(n)$ is the Temperley-Lieb category of index $n$. That is, for any $k,l\in\mathbb N$, the space $Hom(u^{\otimes k},u^{\otimes l})$ can be identified with the abstract vector $D_s(k,l)$ spanned by the Temperley-Lieb diagrams between $2k$ points and $2l$ points, and the categorical operations in $CA_s(n)$ are the usual planar operations in $D_s$ (with the rule that a closed loop corresponds to a multiplicative factor $n$). See e.g. \cite{bbc}.

Also, it is known that the subcategory $CA_k(n)\subset CA_s(n)$ is spanned by a certain subset of diagrams $D_k\subset D_s$, constructed in the following way. For $a,b$ tensor words in $u,\bar{u}$ and for a diagram $T\in D_s(|a|,|b|)$, where $|.|$ is the lenght of words, we have $T\in D_k(a,b)$ provided that the following happens: when putting $a,b$ on the two rows of points of $T$, with the replacements $u\to xy,\bar{u}\to yx$, where $x,y$ are two colors, the strings of $T$ have to match the colors.

Summarizing, we have a diagrammatic description of the categories on the left and on the right, and we want to compute the categories in the middle:
$${\rm span}(D_k)\subset CA_h^\infty(n)\subset CA_h^s(n)\subset {\rm span}(D_s)$$

 The idea would be to prove that the categories in the middle are spanned by certain sets of diagrams, say $D_h^\infty$ and $D_h^s$. 

{\bf Step 3.} In order to compute $D_h^\infty$, we use Woronowicz's Tannakian duality \cite{w2}, which allows one to translate algebraic relations into categorical relations. We know that $A_h^\infty(n)$ is the quotient of $A_k(n)$ by the relations making the elements $u_{ij}$ normal, so what we have to do first is to find a diagrammatic formulation of these normality conditions. Consider the following diagram in $D_s(2,2)$:
$$P=|{\:}^\cup_\cap\,|$$

According to the general rules in \cite{bbc} for diagrammatic calculus for $A_s(n)$, we have the following formula for this diagram, viewed as an operator:
$$P=\sum_ie_{ii}\otimes e_{ii}$$

Now let $u$ be the fundamental corepresentation of $A_k(n)$. We have:
\begin{eqnarray*}
(P\otimes 1)(u\otimes\bar{u})
&=&\left(\sum_ie_{ii}\otimes e_{ii}\otimes 1\right)\left(\sum_{ijkl}e_{ij}\otimes e_{kl}\otimes u_{ij}u_{kl}^*\right)\\
&=&\sum_{ijl}e_{ij}\otimes e_{il}\otimes u_{ij}u_{il}^*\\
&=&\sum_{ij}e_{ij}\otimes e_{ij}\otimes u_{ij}u_{ij}^*
\end{eqnarray*}

Once again by using the defining relations for $A_k(n)$, we have:
\begin{eqnarray*}
(\bar{u}\otimes u)(P\otimes 1)
&=&\left(\sum_{ijkl}e_{ij}\otimes e_{kl}\otimes u_{ij}^*u_{kl}\right)\left(\sum_je_{jj}\otimes e_{jj}\otimes 1\right)\\
&=&\sum_{ijk}e_{ij}\otimes e_{kj}\otimes u_{ij}^*u_{kj}\\
&=&\sum_{ij}e_{ij}\otimes e_{ij}\otimes u_{ij}^*u_{ij}
\end{eqnarray*}

We conclude that the $n^2$ normality conditions for the generators $u_{ij}\in A_k(n)$ are equivalent to the following condition:
$$P\in Hom(u\otimes\bar{u},\bar{u}\otimes u)$$

Now by applying Tannakian duality, we get:
$$CA_h^\infty(n)={\rm span}<D_k,P>$$

Thus we have proved a result announced in Step 2, namely that $CA_h^\infty(n)$ is spanned by diagrams. In the sense of \cite{bbc}, this means that $A_h^\infty(n)$ is free.

{\bf Step 4.} We compute now explicitely $CA_h^\infty(n)$.

We define a subset $D_h^\infty\subset D_s$ in the following way. For $a,b$ tensor words in $u,\bar{u}$ and for a diagram $T\in D_s(|a|,|b|)$, where $|.|$ is the lenght of words, we have $T\in D_h^\infty(a,b)$ provided that the following happens: when collapsing consecutive neighbors of $T$, as to get a noncrossing partition $\tilde{T}$, then putting the words $a,b$ on the points of $\tilde{T}$, each block has the same number of $u$ and $\bar{u}$.

Here the collapsed diagram $\tilde{T}$ is best thought as a being noncrossing partition of points on a circle (with a marked point). Another approach is by using the formalism of annular noncrossing partitions of Mingo and Nica \cite{mn}. 

We have by definition embeddings as follows:
$$D_k\subset D_h^\infty\subset D_s$$

Consider now the diagram $P=|{\:}^\cup_\cap\,|$. When performing the collapsing operation we get the diagram ${\rm H}$, having a single block. Now when putting the tensor words $u\otimes\bar{u}$ and $\bar{u}\otimes u$ on the points of ${\rm H}$, this unique block contains two $u$'s and two $\bar{u}$'s, so the above condition is satisfied. That is, we have:
$$P\in D_h^\infty(u\otimes\bar{u},\bar{u}\otimes u)$$

It is routine to check that each diagram in $D_h^\infty$ decomposes as a product of diagrams in $D_k$ and of diagrams of the following type:
$$||\ldots|{\:}^\cup_\cap\,|\ldots||$$

These latter diagrams being tensor products of $P$ with the identity, we have $D_h^\infty=<D_k,P>$. Now by combining this with the result in Step 3, we get:
$$CA_h^\infty(n)={\rm span}(D_h^\infty)$$

Summarizing, we have a now a fully satisfactory description of $CA_h^\infty(n)$.

{\bf Step 5.} We compute now the category $CA_h^s$, for arbitrary values of $s$. We use the same idea as in Step 3. Consider the following diagram in $D_s(0,s+2)$:
$$Z=|^{\!-\!\!-\!\!-\!\!-\!\!-\!\!-\!\!-\!\!-\!\!-\!}|\hskip-14mm{\ }_{\cap\cap\ldots\cap\cap}$$

When viewed as an operator, this diagram is uniquely determined by its value $\xi=Z(1)$, which is a vector in $(\mathbb C^n)^{\otimes s+2}$, given by the following formula:
$$\xi=\sum_{i}e_i^{\otimes s+2}$$

Now let $u$ be the fundamental corepresentation of $A_h^\infty(n)$. We have:
\begin{eqnarray*}
(u^{\otimes s+1}\otimes\bar{u})(\xi\otimes 1)
&=&\sum_{i_1\ldots i_nj}e_{i_1}\otimes\ldots\otimes e_{i_{s+2}}\otimes u_{i_1j}\ldots u_{i_{s+1}j}u_{i_{s+2}j}^*\\
&=&\sum_{ij}e_i\otimes\ldots\otimes e_i\otimes u_{ij}^{s+1}u_{ij}^*\\
&=&\sum_ie_i^{\otimes s+2}\otimes\left(\sum_ju_{ij}^{s+1}u_{ij}^*\right)
\end{eqnarray*}

This shows that $\xi$ is a fixed vector of the corepresentation $u^{s+1}\otimes\bar{u}$ if and only if the following condition is satisfied, for any $i$:
$$\sum_ju_{ij}^{s+1}u_{ij}^*=1$$

Now since for $i$ fixed the supporting projections $p_{ij}=u_{ij}u_{ij}^*$ are pairwise orthogonal and sum up to 1, it is routine to check that this condition is equivalent to $u_{ij}^s=p_{ij}$, for any $j$. Thus, by getting back to the diagram $Z$, the collection of conditions $u_{ij}^s=p_{ij}$ on the generators $u_{ij}\in A_h^\infty(n)$ is equivalent to:
$$Z\in Hom(1,u^{\otimes s+1}\otimes\bar{u})$$

Now by applying Tannakian duality, we get:
$$CA_h^s(n)={\rm span}<D_h^\infty,Z>$$

Thus we have reached to a similar conclusion to the one at the end of Step 3.

{\bf Step 6.} We compute now explicitely $CA_h^s(n)$. Let us point out first that this category is already known at $s=1,2,\infty$, from \cite{bc2}, \cite{bbc} and from Step 4. 

We define a subset $D_h^s\subset D_s$ in the following way. For $a,b$ tensor words in $u,\bar{u}$ and for a diagram $T\in D_s(|a|,|b|)$, where $|.|$ is the lenght of words, we have $T\in D_h^s(a,b)$ provided that the following happens: when putting the words $a,b$ on the points of the corresponding noncrossing partition $\tilde{T}$, each block has the same number of $u$ and $\bar{u}$, modulo $s$.

We have by definition embeddings as follows:
$$D_h^\infty\subset D_h^s\subset D_s$$

Consider now the diagram $Z$. When performing the collapsing operation we get the diagram $I_{s+2}$ having $s+2$ legs and a single block. Now when putting the word $u^{\otimes s+1}\otimes\bar{u}$ on the points of $I_{s+2}$, this unique block contains $s+1$ copies of $u$ and one copy of $\bar{u}$, so the above condition is satisfied. That is, we have:
$$Z\in D_h^s(1,u^{\otimes s+1}\otimes\bar{u})$$

It is routine to check that each diagram in $D_h^s$ decomposes as a product of diagrams in $D_h^\infty$ and of diagrams of the following type:
$$||\ldots|^{\!-\!\!-\!\!-\!\!-\!\!-\!\!-\!\!-\!\!-\!\!-\!}|\hskip-14mm{\ }_{\cap\cap\ldots\cap\cap}\ \ldots||$$

These latter diagrams being tensor products of $Z$ with the identity, we have $D_h^s=<D_h^\infty,Z>$. Now by combining this with the result in Step 5, we get:
$$CA_h^s(n)={\rm span}(D_h^s)$$

This finishes the categorial computations of the present proof.

{\bf Step 7.} We are now in position of proving the integration results.

We recall from the previous step that the space of fixed vectors of a $k$-fold tensor product $a$ between $u,\bar{u}$ can be identified with the abstract vector space spanned by the set $P_h^s(a)$ of noncrossing partitions of $\{1,\ldots,k\}$, having the following property: when putting the word $a$ on the points of the partition, each block has to contain the same number of $u$ and $\bar{u}$, modulo $s$.

By  \cite{w1} the $*$-moments of $\chi_1$ are the number of fixed points of the tensor products between $u$ and $\bar{u}$, which are in turn equal to the number of diagrams in $P_h^s$. That is, if $e_1,\ldots,e_k\in\{1,*\}$ are exponents, and $a=(u_{ij}^{e_1})\otimes\ldots\otimes (u_{ij}^{e_k})$ is the corresponding tensor product between $u,\bar{u}$, then:
\begin{eqnarray*}
\int\chi_1^{e_1}\ldots\chi_1^{e_k}
&=&\int\chi(a)\\
&=&\dim Hom(1,a)\\
&=&\# P_h^s(a)
\end{eqnarray*}

The idea will be that, by performing a computation using free cumulants, these numbers will turn to be as well the $*$-moments of $\tilde{\pi}_{s1}$.

{\bf Step 8.} We complete now the proof at $t=1$.

We recall from section 7 that if $\alpha_1,\ldots,\alpha_s$ are free free Poisson variables of parameter $1/s$ and $w=e^{2\pi i/s}$, then the following variable has law $\tilde{\pi}_{s1}$: 
$$\alpha=\sum_{l=1}^s w^l\alpha_l$$

We can compute the $*$-moments of this variable by using free cumulants. With standard notations from \cite{ns}, we have:
\begin{eqnarray*}
\int\alpha^{e_1}\ldots\alpha^{e_k}
&=&\sum_{p\in NC(k)}K_p(\alpha^{e_1},\ldots,\alpha^{e_k})\\
&=&\sum_{p\in NC(k)}\sum_{i_1\ldots i_k=1}^sK_p((w^{i_1}\alpha_{i_1})^{e_1},\ldots,(w^{i_k}\alpha_{i_k})^{e_k})\\
&=&\sum_{p\in NC(k)}\sum_{i_1\ldots i_k=1}^sw^{i_1\varepsilon_1+\ldots+i_k\varepsilon_k}K_p(\alpha_{i_1},\ldots,\alpha_{i_k})
\end{eqnarray*}

Here the signs $\varepsilon_i\in\{1,-1\}$ come from the exponents $e_i\in\{1,*\}$.

We use now Speicher's result that the mixed cumulants vanish \cite{s1}. This shows that for a nonzero term in the above sum, the corresponding indices $i_1,\ldots,i_k$ must be constant over the blocks of $p$. Now by factoring each cumulant on the right as a product over the blocks $b=\{b_1,\ldots,b_r\}$ of $p$, we get:
\begin{eqnarray*}
\int\alpha^{e_1}\ldots\alpha^{e_k}
&=&\sum_{p\in NC(k)}\prod_{b\in p}\sum_{i=1}^sw^{i\varepsilon_{b_1}+\ldots+i\varepsilon_{b_r}}K_b(\alpha_1,\ldots,\alpha_1)\\
&=&\sum_{p\in NC(k)}\prod_{b\in p}\sum_{i=1}^s\left(w^{\varepsilon_{b_1}+\ldots+\varepsilon_{b_r}}\right)^iK_b(\alpha_1,\ldots,\alpha_1)\\
&=&\sum_{p\in NC(k)}\prod_{b\in p}(s|\varepsilon_{b_1}+\ldots+\varepsilon_{b_r})sK_b(\alpha_1,\ldots,\alpha_1)
\end{eqnarray*}

Here the symbol $(s|m)$ is given by $(s|m)=1$ if $s|m$, and $(s|m)=0$ if not.

Now, given a partition $p$, in order for its contribution to the above $*$-moment to be nonzero, we must have $s|\varepsilon_{b_1}+\ldots+\varepsilon_{b_r}$ for any block $b\in p$. But this is the same as saying that when putting the word $a$ on the points of $p$, each block of $p$ contains the same number of $u$'s and $\bar{u}$'s, modulo $s$, which is by definition equivalent to $p\in P_h^s(a)$. Thus we have:
$$\int\alpha^{e_1}\ldots\alpha^{e_k}
=\sum_{p\in P_h^s(a)}\prod_{b\in p}sK_b(\alpha_1,\ldots,\alpha_1)$$

Now by general results in \cite{ns}, each of the numbers on the right is $s(1/s)=1$, so the above $*$-moment equals $\# P_h^s(a)$. This finishes the proof at $t=1$. 

Observe that we have ${\rm law}(\chi_1)=\tilde{\pi}_{s1}$, independently of $n\geq 4$. The fact that the convergence is stationary is not surprising, in view of Proposition 10.1 (2).

{\bf Step 9.} We discuss now the general case $t>0$. Here the convergence ${\rm law}(\chi_t)\to\tilde{\pi}_{st}$ will come no longer from a stationary sequence, and we have to use a more technical argument, based on the Weingarten formula:
$$\int u_{i_1j_1}^{e_1}\ldots u_{i_kj_k}^{e_k}=\sum_{p,q\in P_h^s(a)}\delta_{pi}\delta_{qj}W_{an}(p,q)$$

Here we use exponents $e_1,\ldots,e_k\in\{1,*\}$, and $a=(u_{ij}^{e_1})\otimes\ldots\otimes (u_{ij}^{e_k})$ is the corresponding tensor product between $u,\bar{u}$. The delta symbols, equal to 0 or 1, represent the couplings between diagrams and multi-indices, and $W_{an}$ is the Weingarten matrix, obtained as inverse of the Gram matrix. See \cite{bc1}, \cite{bbc}.

Now once again by general arguments developed in \cite{bc1}, \cite{bc2}, \cite{bbc}, the Weingarten formula leads to the following formula for the asymptotic $*$-moments of $\chi_t$:
$$\lim_{n\to\infty}\int\chi_t^{e_1}\ldots\chi_t^{e_k}=\sum_{p\in P_h^s(a)}t^{|p|}$$

Here we use, as above, exponents $e_1,\ldots,e_k\in\{1,*\}$, along with the corresponding tensor product $a=(u_{ij}^{e_1})\otimes\ldots\otimes (u_{ij}^{e_k})$ between $u,\bar{u}$. As for the exponent $|p|$ on the right, this is the number of blocks of $p$. See \cite{bbc}.

At the level of modified free Bessel laws now, what changes when making the replacement $\tilde{\pi}_{s1}\to\tilde{\pi}_{st}$ is that the variables $\alpha_1,\ldots,\alpha_s$ become now free Poisson variables of parameter $t/s$. Thus in the cumulant computation in Step 8 what changes is the contribution of the partitions: instead of a product of numbers $s(1/s)=1$, we have now a product of numbers $s(t/s)=t$. We get:
$$\int\alpha^{e_1}\ldots\alpha^{e_k}=\sum_{p\in P_h^s(a)}t^{|p|}$$

Summarizing, in both computations each partition contributes now with an additive factor $t^b$, where $b$ is the number of blocks, and we are done.
\end{proof}

\end{document}